\documentclass[a4paper, 9pt]{article}
\usepackage{graphics}
\usepackage[dvipdfmx]{hyperref}
\usepackage[backend=bibtex]{biblatex}
\bibliography{kmref.bib}
\usepackage{color}

\usepackage[normalem]{ulem}
\usepackage{lscape}
\usepackage{latexsym}
\usepackage{amsmath,verbatim}
\usepackage{graphics}
\usepackage{amsthm}
\usepackage{amssymb}
\usepackage[mathscr]{euscript}
\usepackage{yfonts}
\usepackage{makeidx}
\usepackage{stmaryrd}
\usepackage{multicol}
\usepackage{bm}
\usepackage[all]{xy}
\numberwithin{equation}{section}
\newtheorem{thm}{Theorem}[section]
\newtheorem{prop}[thm]{Proposition}
\newtheorem{lem}[thm]{Lemma}
\newtheorem{cor}[thm]{Corollary}
\newtheorem{claim}{Claim}{\bf}{\it}

\newtheorem{fthm}{Theorem}{\bf}{\it}
{\bf}{\it}
{\bf}{\it}
{\bf}{\it}
{\bf}{\it}

\theoremstyle{definition}
\newtheorem{defn}[thm]{Definition}

{\bf}{\rm}

\theoremstyle{remark}

\newtheorem{rem}[thm]{Remark}
\newtheorem{frem}[fthm]{Remark}{\bf}{\it}

\newtheorem{definition and corollary}[thm]{Definition and Corollary}

{\it}{\rm}

\newcommand{\cA}{{\mathcal A}}

\newcommand{\al}{\alpha}

\newcommand{\C}{{\mathbb C}}

\newcommand{\cP}{{\mathcal P}}

\newcommand{\la}{\lambda}
\newcommand{\La}{\Lambda}

\newcommand{\Spec}{\mbox{\rm Spec}}

\newcommand{\tI}{\mathtt{I}}

\newcommand{\Q}{\mathbb{Q}}

\newcommand{\Z}{\mathbb{Z}}

\title{Symmetric functions and Springer representations}
\author{Syu \textsc{Kato}\footnote{Department of Mathematics, Kyoto University, Oiwake Kita-Shirakawa Sakyo Kyoto 606-8502 JAPAN \tt{E-mail:syuchan@math.kyoto-u.ac.jp}}}

\begin{document}
\maketitle

\begin{center}
{\bf Dedicated to the memory of Tonny Albert Springer}
\end{center}

\medskip

%

\begin{abstract}
The characters of the (total) Springer representations are identified with the Green functions by Kazhdan [Israel J. Math. {\bf 28} (1977)], and the latter are identified with Hall-Littlewood's $Q$-functions by Green [Trans. Amer. Math. Soc. (1955)]. In this paper, we present a purely algebraic proof that the (total) Springer representations of $\mathop{GL} ( n )$ are $\mathrm{Ext}$-orthogonal to each other, and show that it is compatible with the natural categorification of the ring of symmetric functions.
\end{abstract}

\section*{Introduction}
Let $G$ be a connected reductive algebraic group over an algebraically closed field with a Borel subgroup $B$. Let $W$ be the Weyl groups of $G$, and let $\mathcal N \subset \mathrm{Lie} \, G$ denote the variety of nilpotent elements. The cohomology of a fiber of the Springer resolution
$$\mu : T^* ( G / B ) \longrightarrow \mathcal N,$$
affords a representation of $W$. This is widely recognized as the Springer representation \cite{Spr76}, and it is proved to be an essential tool in representation theory of finite and $p$-adic Chevalley groups \cite{Lus84,KL87,Lus88,Lus90b,Kat09}. Here and below, we understand that the Springer representation refers to the {\it total} cohomology of a Springer fiber instead of the top cohomology, commonly seen in the literature. 

In \cite{Kat15}, we found a module-theoretic realization of Springer representations that is axiomatized as Kostka systems. For $G = \mathop{GL} ( n )$, it takes the following form: Let
$$A = A_n := \C \mathfrak S_n \ltimes \C [X_1,\ldots,X_n]$$
be a graded ring obtained by the smash product of the symmetric group $\mathfrak S_n$ and a polynomial algebra $\C [X_1,\ldots,X_n]$ such that $\deg \, \mathfrak S_n = 0$ and $\deg \, X_i = 1$ ($1 \le i \le n$). Let $A \mathchar`-\mathsf{gmod}$ be the category of finitely generated graded $A$-modules. Let $\mathrm{hom}_{A}$, $\mathrm{end}_{A}$, and $\mathrm{ext}_{A}$ denote the graded versions of $\mathrm{Hom}_{A}$, $\mathrm{End}_{A}$, and $\mathrm{Ext}_{A}$, respectively. The set of simple graded $A$-modules is parametrized by $\mathsf{Irr} \, \mathfrak S_n$ (up to grading shift), and is denoted as $\{L_\la\}_{\la \in \mathsf{Irr} \, \mathfrak S_n}$. We have a projective cover $P_\la \rightarrow L_\la$ as graded $A$-modules.

\begin{fthm}\label{forth}
For each $\la \in \mathsf{Irr} \, \mathfrak S_n$, we have two modules $\widetilde{K}_\la$ and $K_{\la}$ in $A_n \mathchar`-\mathsf{gmod}$ with the following properties:
\begin{enumerate}
\item We have a sequence of $A_n$-module surjections $P_\la \rightarrow\!\!\!\!\!\rightarrow \widetilde{K}_\la \rightarrow\!\!\!\!\!\to K_\la \rightarrow\!\!\!\!\!\to L_\la$, where the first map is obtained by annihilating all graded Jordan-H\"older components $L_{\mu}$ such that $\mu \not\ge \la$ with respect to the dominance order on $\mathsf{Irr} \, \mathfrak S_n$;
\item The graded ring $\mathrm{end}_A ( \widetilde{K}_\la )$ is a polynomial ring. The $($unique$)$ graded quotient $\mathrm{end}_A ( \widetilde{K}_\la ) \to \C_0 \cong \C$ yields $K_\la \cong \C_0 \otimes_{\mathrm{end}_A ( \widetilde{K}_\la )} \widetilde{K}_\la$;
\item We have the following $\mathrm{ext}$-orthogonality:
$$\mathrm{ext}^i_{A}(\widetilde{K}_\la,K_\mu^*)\cong \C^{\oplus\delta_{i,0}\delta_{\la, \mu}}.$$
\end{enumerate}
\end{fthm}

\begin{frem}
If we identify $\la \in \mathsf{Irr} \, \mathfrak S_n$ with a partition, and hence with a nilpotent element $x_\la \in \mathcal N \subset \mathfrak{gl} ( n, \C )$ via the theory of Jordan normal form, then we have
$$K_\la \cong H^\bullet ( \mu^{-1} ( x_{\la} ), \C ) \hskip 5mm \text{and} \hskip 5mm \widetilde{K}_\la \cong H^\bullet_{\mathrm{Stab}_{\mathop{GL} ( n, \C )} (x_\la)} ( \mu^{-1} ( x_{\la} ), \C )$$
with a suitable adjustment of conventions (\cite{Kat15,Kat17}).
\end{frem}

Theorem \ref{forth} follows from works of many people (\cite{Hai01,Hai03,Xi02,KX12,Kle14,CI15,FKM19}) in several different ways as well as an exact account (\cite{Kat15, Kat17}) that works for an arbitrary $G$. All of these proofs utilize some structures (geometry, cells, or affine Lie algebras) that is hard to see in the category of graded $A$-modules.

The main goal of this paper is to give a new proof of Theorem \ref{forth} based on a detailed analysis of $K_{\la}^*$ due to Garsia-Procesi \cite{GP92} and some algebraic results from \cite{Kle14, Kat15}. This completes author's attempt \cite[Appendix A]{Kat15} to give a proof of Theorem \ref{forth} inside the category of graded $A$-modules.
 
As a byproduct, we obtain an interesting consequence: We call $M \in A \mathchar`-\mathsf{gmod}$ (resp. $M \in A \boxtimes A \mathchar`-\mathsf{gmod}$) to be $\Delta$-filtered (resp. $\overline{\Delta}$-filtered) if $M$ admits a decreasing separable filtration (resp. finite filtration) whose associated graded is isomorphic to the direct sum of $\{\widetilde{K}_\la\}_\la$ (resp. direct sum of $\{L_{\la} \boxtimes K_\mu\}_{\la,\mu}$) up to grading shifts.

\begin{fthm}[$\doteq$ Theorem \ref{indres}]\label{ffilt}
The induction of graded $A$-modules sends the external tensor product of $P_{\la}$ and a $\Delta$-filtered module to a $\Delta$-filtered module. Dually, the restriction of graded $A_n$-modules sends a $\overline{\Delta}$-filtered module of $A_n$ $(= A_0 \boxtimes A_n)$ to a $\overline{\Delta}$-filtered module of $A_{r} \boxtimes A_{n-r}$ $(0 \le r \le n)$.
\end{fthm}

Recall that the graded modules
$$\bigoplus_{n \ge 0} K ( A_n \mathchar`-\mathsf{gmod}) \subset \Q (\!(q)\!) \otimes _\Z \bigoplus_{n \ge 0} K ( \mathfrak S_n \mathchar`-\mathsf{mod}),$$
are Hopf algebras by Zelevinsky \cite{Zel81}, that is identified with the ring $\La$ of symmetric functions up to scalar extensions (Theorem \ref{Sind}). In particular, this ring is equipped with four bases $\{s_\la\}_\la, \{Q^{\vee}_\la\}_\la, \{Q_\la\}_\la$, and $\{ S_\la \}_\la$, usually referred to as the Schur functions, the Hall-Littlewood $P$-functions, the Hall-Littlewood $Q$-functions, and the big Schur functions, respectively (\cite{Mac95}). We exhibit a natural character identification (that we call the {\it twisted} Frobenius characteristic)
\begin{equation}
\begin{array}{c|cccc}
\textrm{Modules of }A & P_\la & \widetilde{K}_\la & K_\la & L_\la\\ \hline
\textrm{Basis of }\La & s_\la & Q^{\vee}_\la & Q_\la & S_\la
\end{array}
\end{equation}
that intertwines the products with inductions, and the coproducts with restrictions. (The complete symmetric functions and the elementary symmetric functions are expanded positively by the Schur functions, and hence corresponds to a direct sum of projective modules in this table).

Under this identification, Theorem \ref{ffilt} implies that the multiplication of a Schur function in $\La$
exhibits positivity with respect to the Hall-Littlewood functions (Corollary \ref{HLnum}). In addition, we deduce a homological interpretation of skew Hall-Littlewood functions (Corollary \ref{skewpos}).

In a sense, our exposition here can be seen as a direct approach to an algebraic avatar of the Springer correspondence. We note that interpreting sheaves appearing in the Springer correspondence as constructible functions produces totally different algebraic avatar of the Springer correspondence via Hall algebras (as pursued in Shimoji-Yanagida \cite{SY20}). Although our Hopf algebra structure is closely related to the Heisenberg categorification (cf. \cite{CL12}), the author was not able to find a result of this kind in the literature. Nevertheless, he plans to write a follow-up paper that covers the relation with the Heisenberg categorification in an occasion.

Finally, the author was very grateful to find related \cite{SZ84} during the preparation of this paper.

\section{Preliminaries}\label{sec:prelim}

A vector space is always a $\C$-vector space, and a graded vector space refers to a $\Z$-graded vector space whose graded pieces are finite-dimensional and its grading is bounded from the below. Tensor products are taken over $\C$ unless stated otherwise. We define the graded dimension of a graded vector space as
$$\mathrm{gdim} \, M := \sum_{i\in \Z} q^i \dim _{\C} M_i \in \Q (\!(q)\!).$$
In case $\dim \, M < \infty$, we set $M^* := \bigoplus_{i \in \Z}( M^* )_i$, where $(M^*)_i := ( M_{-i} )^*$ for each $i \in \Z$. We set $[n]_q := \frac{1-q^n}{1-q}$ for each $n\in \Z_{\ge 0}$.

For a $\C$-algebra $A$, let $A \mathchar`-\mathsf{mod}$ denote the category of finitely generated left $A$-modules. If $A$ is a graded algebra in the sense that $A = \bigoplus_{i \in \Z_{\ge 0}} A_i$ and $A_i A_j \subset A_{i+j}$ ($i,j \in \Z_{\ge 0}$), then we denote by $A \mathchar`-\mathsf{gmod}$ the category of finitely generated graded $A$-modules. We also have a full subcategory $A \mathchar`-\mathsf{fmod}$ of $A \mathchar`-\mathsf{gmod}$ consisting of finite-dimensional modules.

For a graded algebra $A$, the category $A \mathchar`-\mathsf{gmod}$ admits an autoequivalence $\left< n \right>$ for each $n \in \Z$ such that $M = \bigoplus_{i \in \Z} M_i$ is sent to $M \left< n \right> := \bigoplus_{i \in \Z} ( M \left< n \right> )_i$, where $( M \left< n \right> )_i = M_{i - n}$. For $M, N \in A \mathchar`-\mathsf{gmod}$, we set
\begin{align*}
\mathrm{hom}_A ( M, N ) & \, := \bigoplus_{j \in \Z} \mathrm{hom}_{A} ( M , N )_j, \hskip 3mm \mathrm{hom}_{A} ( M , N )_j := \mathrm{Hom}_{A \mathchar`-\mathsf{gmod}} ( M\left< j \right>, N  ),\\
\mathrm{ext}^i_A ( M, N ) & \, := \bigoplus_{j \in \Z} \mathrm{ext}^i_{A} ( M, N )_j, \hskip 3mm \mathrm{ext}^i_{A} ( M, N )_j := \mathrm{Ext}^i_{A \mathchar`-\mathsf{gmod}} ( M \left< j \right>, N  ).
\end{align*}
In particular, $\mathrm{hom}_A ( M, N )$ and $\mathrm{ext}^\bullet_A ( M, N )$ are graded vector spaces if $\dim \, A_i < \infty$ for each $i \in \Z_{\ge 0}$. Moreover, $\mathrm{hom}_A ( M, N )_j$ consists of graded $A$-module homomorphisms that raise the degree by $j$.

For $M \in A \mathchar`-\mathsf{gmod}$, the head of $M$ (that we denote by $\mathsf{hd}\, M$) is the maximal semisimple graded quotient of $M$, and the socle of $M$ (that we denote by $\mathsf{soc}\, M$) is the maximal semisimple graded submodule of $M$.

For a decreasing filtration
$$M = F_0 M \supset F_1 M \supset F_2 M\supset \cdots$$
of graded vector spaces, we define its $k$-th associated graded piece as $\mathrm{gr}^F_k M := F_k M / F_{k+1} M$ ($k \ge 0$). We call such a filtration separable if $\bigcap _{k\ge 0} F_k M = \{ 0 \}$.

For an exact category $\mathcal C$, let $[\mathcal C]$ denote its Grothendieck group. For $M \in \mathcal C$, we have its class $[M] \in [\mathcal C]$. In case $\mathcal C$ admits the grading shift functor $\left< n \right>$ ($n\in \Z$), an element $f = \sum_{n}a_n q^n \in \Z[q^{\pm 1}]$ ($a_n \in \Z_{\ge 0}$) defines the direct sum
$$M^{\oplus f} := \bigoplus_{n \in \Z} \left( M \left< n \right> \right)^{\oplus a_n} \hskip 5mm M \in \mathcal C.$$
We may represent a number that is not important by $\star \in \Z [q^{\pm 1}]$.

\subsection{Partitions and the ring of symmetric functions}
We employ \cite{Mac95} as the general reference about partitions and symmetric functions. We briefly recall some key notions there. The set of partitions is denoted by $\cP$, and the set of partitions of $n$ $(\in \Z_{\ge 0})$ is denoted by $\cP_n$. Each of $\cP_n$ is equipped with a partial order $\le$ such that $(n)$ is the largest element. We extend the order $\le$ to the whole $\cP$ by declaring that elements of $\cP_n$ and $\cP_m$ are comparable only if $n = m$. Let $m_i ( \la )$ be the multiplicity of $i$, let $\ell ( \la )$ be the partition length, and let $|\la|$ be the partition size of $\la \in \cP$. The conjugate partition of $\la \in \cP$ is denoted by $\la'$. We set
$$n ( \la ) := \sum_{i \ge 1} (i-1)\la_i = \sum_{i \ge 1} \left( \begin{matrix} \la'_i\\2\end{matrix}\right).$$
For $\la \in \cP_n$ and $1 \le j \le \ell ( \la )+1$, let $\la^{(j)} \in \cP_{n}$ be the partition of $(n+1)$ obtained by rearranging $\{\la_i\}_{i \neq j} \cup \{\la_j + 1\}$, and for $1 \le j \le \ell ( \la )$, we set $\la_{(j)}$ be the partition of $(n-1)$ obtained by rearranging $\{\la_i\}_{i \neq j} \cup \{\la_j - 1\}$. We set
$$b_\la (q) = \prod_{j \ge 1} \left( (1-q) \cdots (1-q^{m_j ( \la )})\right).$$

Let $\La$ be the ring of symmetric functions with their coefficients in $\Z$. Let $\La_{q}$ be its scalar extension to $\Q (\!( q )\!)$. We have direct sum decompositions $\La = \bigoplus_{n \ge 0} \La_n$ and $\La_q = \bigoplus_{n \ge 0} \La_{q,n}$ into the graded components. The ring $\La$ is equipped with four distinguished bases
$$\{ h _\la \}_{\la \in \cP}, \hskip 3mm \{ s_\la \}_{\la \in \cP}, \hskip 3mm \{ e_\la \}_{\la \in \cP}, \hskip 3mm \text{and} \hskip 3mm \{ m_{\la} \}_{\la \in \cP},$$
called (the sets of) complete symmetric functions, Schur functions, elementary symmetric functions, and monomial symmetric functions, respectively. We have equalities
$$h_1 = s_{(1)} = e_1 = m_{(1)}, \hskip 3mm h_n = s_{(n)}, \hskip 3mm \text{and} \hskip 3mm e_n = s_{(1^n)} \hskip 5mm n \in \Z_{>0}.$$

We have a symmetric inner product $(\bullet,\bullet)$ on $\La$ such that
$$(s_\la, s_{\mu}) = ( h_\la, m_\mu ) = \delta_{\la,\mu} \hskip 5mm \la, \mu \in \cP.$$

The ring $\La$ has a structure of a Hopf algebra with the coproduct $\Delta$ satisfying
$$\Delta ( h_n ) = \sum_{i+j=n} h_i \otimes h_j, \hskip 3mm \text{and} \hskip 3mm \Delta ( e_n ) = \sum_{i+j = n} e_i \otimes e_j$$
and the antipode $S$ satisfying
$$S ( h_n ) = (-1)^n e_n, \hskip 3mm \text{and} \hskip 3mm S ( e_n ) = (-1)^n h_n.$$
The antipode $S$ preserves the inner product $(\bullet,\bullet)$.

\subsection{Zelevinsky's picture for symmetric groups}

For a (not necessarily non-increasing) sequence $\la = ( \la_1,\la_2,\ldots) \in \Z_{\ge 0}^{\infty}$ such that $\sum_{j} \la_j = n$, we define the subgroup
$$\mathfrak S_\la := \prod_{j \ge 1} \mathfrak S_{\la_j} \subset \mathfrak S_n.$$
We usually omit $0$ in $\la = ( \la_1,\la_2,\ldots)$. Each $\la \in \cP_n$ defines an irreducible representation of $L_\la$ of $\mathfrak S_n$. We normalize $L_{\la}$ such that
$$L_{(n)} \cong \mathsf{triv}, \hskip 3mm \text{and} \hskip 3mm L_{(1^n)} \cong \mathsf{sgn}.$$
For $0 < r < n$, we have induction/restriction functors
\begin{align*}
\mathrm{Ind}_{r,n-r} : \,  & \C \mathfrak S_{(r,n-r)} \mathchar`-\mathsf{mod} \ni ( M, N ) \mapsto \C \mathfrak S_{n} \otimes_{\C \mathfrak S_{(r,n-r)}} ( M \boxtimes N) \in \mathfrak S_{n} \mathchar`-\mathsf{mod}\\
\mathrm{Res}_{r,n-r} : \, & \C \mathfrak S_{n} \mathchar`-\mathsf{mod} \longrightarrow \C \mathfrak S_{(r,n-r)} \mathchar`-\mathsf{mod},
\end{align*}
where the latter is the natural restriction. They induce corresponding maps between the Grothendieck groups that we denote by the same letter.

\begin{thm}[Zelevinsky \cite{Zel81}]\label{Sind}
We have a $\Z$-module isomorphism
$$\Psi_0 : \bigoplus_{n \ge 0} [\C \mathfrak S_{n} \mathchar`-\mathsf{mod}] \ni [L_\la] \mapsto s_\la \in \La.$$
with the following properties: For $M \in \C \mathfrak S_{r} \mathchar`-\mathsf{mod}$ and $N \in \C \mathfrak S_{n} \mathchar`-\mathsf{mod}$, we have
$$
\Psi_0 ( \mathrm{Ind}_{r,n} \, [ M \boxtimes N ] ) = \Psi_0 ( [M] ) \cdot \Psi_0 ( [N] ), \hskip 5mm \sum_{s=0}^n \Psi_0 ( \mathrm{Res}_{s,n-s} \, [N] ) = \Delta ( [N] ).
$$
In particular, we have
$$
h_r \cdot \Psi_0 ( [N] ) = \Psi_0 ( \mathrm{Ind}_{r,n} \, [ L_{(r)} \boxtimes N ] ), \hskip 5mm  e_r \cdot \Psi_0 ( [N] ) = \Psi_0 ( \mathrm{Ind}_{r,n} \, [ L_{(1^r)} \boxtimes N ] ).
$$
\end{thm}

\subsection{The algebra $A_n$ and its basic properties}

We follow \cite[\S 2]{Kat15} here. We set
$$A _n := \C \mathfrak S_n \ltimes \C [X_1,\ldots, X_n],$$
where $\mathfrak S_n$ acts on the ring $\C [X_1,\ldots, X_n]$ by
$$(w \otimes 1) ( 1 \otimes X_i) = ( 1 \otimes X_{w ( i )} ) ( w \otimes 1 ) \hskip 5mm w \in \mathfrak S_n, 1 \le i \le n.$$
We usually denote $w$ in place of $w \otimes 1$, and $f \in \C [X_1,\ldots,X_n]$ in place of $1 \otimes f$. The ring $A_n$ acquires the structure of a graded ring by
$$\deg \, w = 0, \hskip 3mm \deg \, X_i = 1 \hskip 5mm w \in \mathfrak S_n, 1 \le i \le n.$$
The grading of the ring $A_n$ is non-negative, and the positive degree part $A_n^+ := \bigoplus_{j > 0} A_n^j$ defines a graded ideal such that $A_n / A_n^+ \cong\C \mathfrak S_n\cong A_n^0$. In particular, each $L_\la$ can be understood to be a graded $A_n$-module concentrated in degree $0$.

The assignments $w \mapsto w^{-1}$ ($w \in W$) and $X_i \mapsto X_i$ ($1 \le i \le n$) define an isomorphism $A_n \cong A_n^{op}$. Therefore, if $M \in A_n \mathchar`-\mathsf{fmod}$, then $M^*$ acquires the structure of a graded $A_n$-module. We have $( L_\la )^* \cong L_\la$ for each $\la \in \cP_n$ as $\mathfrak S_n$ is a real reflection group.

For each $\la \in \cP_n$, we have an idempotent $e_\la \in \C \mathfrak S_n$ such that $L_\la \cong \C \mathfrak S_n e_\la$. We set $P_\la := A_n e_\la$.

\begin{prop}[see \cite{Kat15} \S 2]\label{PLcov}
The modules $\{L_\la \left< j \right>\}_{\la \in \cP_n, j \in \Z}$ is the complete collection of simple objects in $A_n \mathchar`-\mathsf{gmod}$. In addition, $P_\la$ is the projective cover of $L_\la$ in $A_n \mathchar`-\mathsf{gmod}$ for each $\la \in \cP_n$. \hfill $\Box$
\end{prop}

We define
$$\widetilde{K}_\la := \frac{P_\la}{\sum_{\mu \not\ge \la, f \in \hom_A ( P_{\mu}, P_\la )} \mathrm{Im} \, f} \hskip 3mm \text{and} \hskip 3mm K_\la := \frac{\widetilde{K}_\la}{\sum_{j > 0, f \in \hom_A ( P_{\la}, \widetilde{K}_\la )_{j}} \mathrm{Im} \, f}.$$

For each $M \in A \mathchar`-\mathsf{gmod}$, we set
$$[M:L_\la]_q := \mathrm{gdim}\, \mathrm{hom}_A ( P_\la, M ) = \sum_{i \in \Z} q^i \dim \, \mathrm{Hom}_{\mathfrak S_n} ( L_\la, M_i ) \in \Z (\!(q)\!).$$
In case the $q = 1$ specialization of $[M:L_{\la}]_q$ makes sense, we denote it by $[M:L_\la]$.

\begin{lem}[see \cite{Kat15} \S 2]\label{Kmult}
For each $\la \in \cP_n$, we have
$$
[K_\la : L_{\mu}]_q = \begin{cases} 0 & \la \not\le \mu\\ 1 & \la = \mu \end{cases}, \hskip 3mm [\widetilde{K}_\la : L_{\mu}]_q \in \begin{cases} 0 & \la \not\le \mu\\ 1 + q \Z [\![q]\!] & \la = \mu \end{cases}.
$$
\end{lem}

\begin{proof}
Immediate from the definition.
\end{proof}

For $0 \le r \le n$, we consider the subalgebra
$$A_{r,n-r} := \C \mathfrak S_{(r,n-r)} \ltimes \C [X_1,\ldots, X_n] \cong A_r \boxtimes A_{n-r} \subset A_n.$$

We have induction/restriction functors
\begin{align*}
\mathrm{ind} _{r,n-r} & \, : A_{r,n-r} \mathchar`-\mathsf{gmod} \ni M \mapsto A_n \otimes_{A_{r,n-r}}M \in A_{n} \mathchar`-\mathsf{gmod},\\
\mathrm{res} _{r,n-r} & \, : A_{n} \mathchar`-\mathsf{gmod} \longrightarrow A_{r,n-r} \mathchar`-\mathsf{gmod}.
\end{align*}
Since $A_n$ is free of rank $\frac{n!}{r!(n-r)!}$ over $A_{r,n-r}$, we find that the both functors are exact, and preserves finite-dimensionality of the modules. We sometimes omit the functor $\mathrm{res}_{r,n-r}$ from notation in case there are no possible confusion.

We consider the category $\cA := \bigoplus_{n\ge 0} A_{n} \mathchar`-\mathsf{gmod}$. We define
$$\mathrm{ind} := \bigoplus_{r,s} \mathrm{ind}_{r,s} : \cA \times \cA \rightarrow \cA, \hskip 5mm \mathrm{res} := \bigoplus_{r,s} \mathrm{res}_{r,s} : \cA \rightarrow \cA \boxtimes \cA.$$

\begin{lem}\label{id-indres}
We embed $\mathfrak S_n \mathchar`-\mathsf{mod}$ into $A_{n} \mathchar`-\mathsf{gmod}$ by regarding $M \in \mathfrak S_n \mathchar`-\mathsf{mod}$ as a semisimple graded $A_n$-module concentrated in degree $0$ for each $n \in \Z_{\ge 0}$. Then, we have
$$\mathrm{Ind}_{r,s} = \mathrm{ind}_{r,s} \hskip 5mm \text{and} \hskip 5mm \mathrm{Res}_{r,s} = \mathrm{res}_{r,s} \hskip 5mm r, s \in \Z_{\ge 0}$$
on $\bigoplus_{n \ge 0} \mathfrak S_n \mathchar`-\mathsf{mod}$. In particular, $[\cA]$ can be understood as a $($Hopf$)$ subalgebra of $\C (\!(q)\!) \otimes \La = \La_q$ by extending the scalar in Theorem \ref{Sind}. \hfill $\Box$
\end{lem}

The following three theorems are quite well-known to experts.

\begin{thm}[Frobenius-Nakayama reciprocity]\label{fnr}
For $M \in A_{r,n-r} \mathchar`-\mathsf{gmod}$ and $N \in A_n\mathchar`-\mathsf{gmod}$, it holds
$$\mathrm{ext}^k_{A_{n}} ( \mathrm{ind}_{r,n-r}\,M, N )\cong\mathrm{ext}^k_{A_{r,n-r}} ( M, \mathrm{res}_{r,n-r} \, N ) \hskip 5mm k \in \Z.$$
\end{thm}

\begin{proof}
This follows from the fact that $A_n$ is a free $A_{r,n-r}$-module by the classical Frobenius reciprocity as $\mathrm{ind}_{r,n-r}$ sends a projective resolution of $M$ to a projective resolution of $\mathrm{ind}_{r,n-r}\,M$.
\end{proof}

\begin{thm}\label{ad}
For $M, N \in A_n\mathchar`-\mathsf{fmod}$, it holds
$$\mathrm{ext}^k_{A_{n}} ( M, N )\cong\mathrm{ext}^k_{A_{n}} ( N^*, M^* ) \hskip 5mm k \in \Z.$$
\end{thm}

\begin{proof}
We borrow terminology from \cite[\S 2.2]{Gro57}. We have natural isomorphism
$$\mathrm{hom}_{A_{n}} ( M, N )\cong\mathrm{hom}_{A_{n}} ( N^*, M^* ).$$
Since the derived functors of the both sides (defined in an appropriate ambient categories) are $\delta$-functors in each variables, it suffices to see that they are universal $\delta$-functors. By approximating $N$ by its injective envelope (and hence $N^*$ by its projective cover), we find that the both sides are effacable on the second variables. Thus, they must coincide by \cite[2.2.1 Proposition]{Gro57}.
\end{proof}

\begin{thm}\label{gdim}
The global dimension of $A_n$ is finite. In particular, every $M \in A_{n} \mathchar`-\mathsf{gmod}$ admits a graded projective resolution of finite length. 
\end{thm}

\begin{proof}
See McConnell-Robson-Small \cite[7.5.6]{MR01}.
\end{proof}

We have a $\Z [ q^{\pm 1} ]$-bilinear inner product $\left< \bullet, \bullet \right>_{EP}$ on $[\cA]$ prolonging
$$A_{n} \mathchar`-\mathsf{gmod} \times A_{n} \mathchar`-\mathsf{fmod} \ni (M,N) \mapsto \sum_{i \ge 0} (-1)^i \mathrm{gdim} \, \mathrm{ext}^i_{A_n} ( M, N^* )^* \in \Q (\!( q )\!).$$

\begin{lem}
The pairing $\left< \bullet, \bullet \right>_{EP}$ is a well-defined symmetric form on $[\mathcal A]$.
\end{lem}

\begin{proof}
Since the Euler-Poincar\'e form respects the short exact sequences, the form $\left< \bullet, \bullet \right>_{EP}$ must be additive with respect to the both variables.

By the arrangement of duals in the definition of $\left< \bullet, \bullet \right>_{EP}$, we find that replacing $M$ with $M\left< n \right>$ and replacing $N$ with $N \left< n \right>$ both result in multiplying $q^n$ ($n \in \Z$). As the category $\cA$ has finite direct sums, we conclude that $\left< \bullet, \bullet \right>_{EP}$ must be $\Z [q^{\pm 1}]$-bilinear.

We have
$$[A_{n} \mathchar`-\mathsf{gmod}] = \bigoplus_{\la \in \cP_n} \Z[q^{\pm 1}] [P_\la] \subset \bigoplus_{\la \in \cP_n} \Q(\!(q)\!) [L_\la]$$
  by Proposition \ref{PLcov}. In particular, $\left< L_\la,  L_{\mu}\right>_{EP} \in \Q(\!(q)\!)$ ($\la, \mu \in \cP_n$) uniquely determines a well-defined $\Q (\!(q)\!)$-bilinear form $\left< \bullet, \bullet \right>_{EP}$ that restricts to $[\cA]$. It is symmetric by Theorem \ref{ad}.
\end{proof}

\section{Main results}

Keep the setting of the previous section.

\begin{defn}
Fix $0 \le r \le n$. A $\Delta$-filtration (resp. $\overline{\Delta}$-filtration) of $M \in A_{n} \mathchar`-\mathsf{gmod}$ is a decreasing separable filtration
$$M = F_0 M \supset F_1 M \supset F_2 M \supset \cdots$$
of graded $A_n$-modules (resp. graded $A_{r,n-r}$-modules) such that
$$\mathrm{gr}^F_k M \in \{\widetilde{K}_\la \left< m \right> \}_{\la \in \cP_n, m \in \Z} \hskip 4mm \text{(resp. } \mathrm{gr}^F_k M \in \{ L_{\mu} \boxtimes K_\nu \left< m \right> \}_{\mu \in \cP_r, \nu \in \cP_{n-r}, m \in \Z}\text{)}$$
for each $k \ge 0$. In case $M$ admits a $\Delta$-filtration, then we set
$$( M: \widetilde{K}_{\la})_q := \sum_{k = 0}^{\infty} \sum_{m \in \Z} q^m \chi ( \mathrm{gr}^F_k M \cong \widetilde{K}_\la \left< m \right> ),$$
where $\chi ( \mathfrak X )$ takes value $1$ if the proposition $\mathfrak X$ is true, and $0$ otherwise.
\end{defn}

\begin{lem}[\cite{Kat15} \S 2 or \cite{Kle14}]
The multiplicity $( M: \widetilde{K}_{\la} )_q$ does not depend on the choice of $\Delta$-filtration. \hfill $\Box$
\end{lem}

The following theorem is not new (see Remark \ref{hist}). Nevertheless, the author feels it might worth to report a yet another proof based on Garsia-Procesi \cite{GP92}, that differs significantly from other proofs and is carried out within the category of $A_n$-modules:

\begin{thm}\label{main}
Let $\la,\mu \in \cP_n$. We have the following:
\begin{enumerate}
    \item For each $\la \in \cP_n$, the graded ring $\mathrm{end}_{A_n} ( \widetilde{K}_\la )$ is a polynomial ring generated by homogeneous polynomials of positive degrees;
\item The module $\widetilde{K}_\la$ is free over $\mathrm{end}_{A_n} ( \widetilde{K}_\la )$, and we have
$$\C_0 \otimes _{\mathrm{end}_{A_n} ( \widetilde{K}_\la )} \widetilde{K}_\la \cong K_\la,$$
where $\C_0$ is the unique graded one-dimensional quotient of $\mathrm{end}_{A_n} ( \widetilde{K}_\la )$; 
\item We have the $\mathrm{Ext}$-orthogonality:
$$\mathrm{ext}^i_{A_n} ( \widetilde{K}_\la, K_\mu^* ) \cong \C^{\oplus \delta_{\la,\mu} \delta_{i,0}};$$
\item Each $P_\la$ admits a $\Delta$-filtration, and we have $(P_\la:\widetilde{K}_\mu) _q = [K_\mu:L_\la]_q$.
\end{enumerate}
\end{thm}

\begin{proof}
Postponed to \S \ref{mainproof}.
\end{proof}

\begin{rem}\label{hist}
Theorem \ref{main} is originally proved in \cite{Kat15, Kat17} essentially in this form by using the geometry of Springer correspondence (that works for arbitrary Weyl groups with arbitrary cuspidal data). Theorem \ref{main} also follows from results of Haiman \cite{Hai01,Hai03} that employ the geometry of Hilbert schemes of points on $\C^2$. We also have two algebraic proofs of Theorem \ref{main}, one is to use a detailed study of two-sided cells of affine Hecke algebras by Xi \cite{Xi02} together with K\"onig-Xi \cite{KX12} and Kleshchev \cite{Kle14}, and another is an analogous result for affine Lie algebras (Chari-Ion \cite{CI15}) together with Feigin-Khoroshkin-Makedonskyi \cite{FKM19}.
\end{rem}

We exhibit applications of Theorem \ref{main} in \S \ref{indproof}.

\subsection{Garsia-Procesi's theorem}

For each $\tI \subset [1,n]$ and $|\tI| \ge r \ge 1$, let $e_r ( \tI )$ be the $r$-th elementary symmetric function with respect to the variables $\{X_i\}_{i \in \tI}$. For $\la \in \cP_n$, we set
$$d_r ( \la ) := \la'_1 + \cdots + \la'_r \hskip 5mm (1 \le r \le n).$$
We set
$$\mathcal C_\la := \{ e_t ( \tI ) \mid r \ge t \ge r - d_r ( \la ), |\tI| = r, \tI \subset [1,n] \}.$$
Let $I _\la \subset \C [X_1,\ldots,X_n]$ be the ideal generated by $\mathcal C_\la$ (originally introduced in \cite{Tan82}).

\begin{defn}
We set $R_\la := \C [X_1,\ldots,X_n] / I_\la$, and call it the Garsia-Procesi module.
\end{defn}

\begin{lem}[\cite{GP92} \S 3]
The algebra $R_\la$ admits a structure of graded $A_n$-module generated by $L_{(n)}$. In addition, $[R_\la:L_{(n)}] _q = 1$.
\end{lem}

\begin{proof}
Since $R_\la$ is the quotient of $P_{(n)}$, it suffices to see that the ideal $I_\la$ is graded and $\mathfrak S_n$-stable. Since $\mathcal C_\la$ consists of homogeneous polynomials and it is stable under the $\mathfrak S_n$-action, we conclude the first assertion. For the second assertion, it suffices to notice that $\mathcal C_\la$ contains all the elementary symmetric polynomials in $\C [X_1,\ldots,X_n]$, and hence $I_\la$ contains all the positive degree part of $\C [X_1,\ldots,X_n]^{\mathfrak S_n}$.
\end{proof}

\begin{thm}[Garsia-Procesi \cite{GP92} \S 1]\label{GP}
Let $\la \in \cP_n$. The $\C [X_1,\ldots,X_n]$-module $R_\la$ admits a decreasing filtration
\begin{equation}
R_\la = F_0 R_\la \supset F_1 R_\la \supset \cdots \supset F_{\ell ( \la )} R_\la = \{ 0 \}\label{GPfilt}
\end{equation}
such that $\mathrm{gr}^F_j R_\la \cong R_{\la_{(j+1)}} \left< j \right>$ for $0 \le j < \ell ( \la )$. In addition, this filtration respects the $\mathfrak S_{n-1}$-action, and hence can be regarded as an $A_{1,n-1}$-module filtration. \hfill $\Box$
\end{thm}

\begin{thm}[\cite{GP92} Theorem 3.1 and Theorem 3.2]\label{GPprop}
Let $\la \in \cP_n$. It holds:
\begin{enumerate}
\item We have $( R_\la )_{n ( \la ) + 1} =\{ 0 \}$;
\item We have a $\mathfrak S_n$-module isomorphism $R_\la \cong \mathrm{ind}^{\mathfrak S_n} _{\mathfrak S_{\la}} \, \mathsf{triv}$.
\end{enumerate}
In particular, we have $[R_\la:L_{\mu}] \neq 0$ only if $\la \le \mu$. \hfill $\Box$
\end{thm}

In view of \cite[I\!I\!I (2.1)]{Mac95}, we have the Hall-Littlewood $P$- and $Q$- functions in $\La_q$ indexed by $\cP$, that we denote by $Q^{\vee}_\la$ and $Q_\la$, respectively (we changed notation of $P$-functions to $Q^{\vee}$ in order to avoid confusion with projective modules). They satisfy the following relation:
\begin{equation}
  Q^{\vee}_{\la} = b_\la^{-1} Q _\la \in \La_q.\label{QtoQv}
\end{equation}
We also have the big Schur function (\cite[I\!I\!I (4.6)]{Mac95})
$$S_\la := \prod_{i<j} ( 1 - q R_{ij}) Q_\la,$$
where $R_{ij}$ are the raising operators.

\begin{thm}[\cite{GP92} \S 5, particularly (5.24)]\label{GPHL}
For each $\la \in \cP$, the polynomial
$$Q_\la := \sum_{\mu} [K_\la : L_{\mu}]_q \cdot S_{\mu} \in \La_q$$
is the Hall-Littlewood's $Q$-function. \hfill $\Box$
\end{thm}

\begin{thm}[\cite{Mac95} I\!I\!I (4.9)]\label{defHall}
There exists a $\Q ( q )$-linear bilinear form $\left< \bullet, \bullet \right>$ on $\La_q$ $($referred to as the Hall inner product$)$ characterized as
\begin{equation}
\left< Q^{\vee}_{\la}, Q_{\mu} \right> = \delta_{\la,\mu} = \left< S_\la, s_{\mu} \right>\label{Horth}
\end{equation}
for each $\la, \mu \in \cP$. \hfill $\Box$
\end{thm}

%

\begin{lem}\label{Rsoc}
For each $\la \in \cP_n$, we have $[R_\la : L _\la ]_q = q^{n ( \la )}$.
\end{lem}

\begin{proof}
By \cite[p115]{Mac95} and the Frobenius reciprocity, $L_\la$ contains a vector on which $\mathfrak S_{\la'}$ acts by sign representation. Since the Vandermonde determinant offers the minimal degree realization of the sign representations of each $\mathfrak S_{\la'_j}$ ($1 \le j \le \la_1$), we find that $\mathrm{Hom}_{\mathfrak S_n} ( L_\la, ( R_\la )_m ) \neq 0$ only if $m \ge n ( \la )$. It must be strict by Theorem \ref{GPprop} 1).
\end{proof}

\begin{prop}[\cite{Kat15} Theorem A.4 and Corollary A.3]\label{char-K}
We have
$$\mathrm{ext}^1_{A_n} ( K_\la, L_\mu ) = 0 \hskip 5mm \la \not\ge \mu.$$
For each $\la \in \cP_n$, the head of $K_\la$ is $L_\la$, and the socle of $K_\la$ is $L_{(n)} \left< n ( \la )\right>$.
\end{prop}

\begin{proof}
By \cite[Theorem A.4]{Kat15}, the module $K_\la$ is isomorphic to the module $M_{\la}$ constructed there. They have the properties in the assertions by construction and \cite[Theorem A.4]{Kat15}.
\end{proof}

\begin{prop}[De Concini-Procesi \cite{DP81}, Tanisaki \cite{Tan82}]\label{RK-id}
We have an isomorphism $R_\la^* \left< n ( \la ) \right> \cong K_\la$ as graded $A_n$-modules.
\end{prop}

\begin{proof}
By Lemma \ref{Rsoc}, $R _\la^* \left< n ( \la ) \right>$ is a graded $A_n$-module such that $L_\la \subset \mathsf{hd} \, R _\la^* \left< n ( \la ) \right>$ and $[R_\la ^* \left< n ( \la ) \right> : L_{\mu}]_q = 0$ if $\mu \not\ge \la$ and $[R_\la^* \left< n ( \la ) \right> : L_\la]_q = 1$. Thus, we obtain a map $K_\la \rightarrow R _\la^* \left< n ( \la ) \right>$ of graded $A_n$-modules. This map is injective as they share $L_{(n)} \left< n( \la )\right>$ as their socles.

We prove that $K_\la \subset R_\la^* \left< n ( \la ) \right>$ is an equality for every $\la \in \cP_n$ by induction on $n$. The case $n=1$ is clear as the both are $\C$. Thanks to Theorem \ref{GP} and the induction hypothesis, we deduce that a (graded) direct summand of the head of $R_\la^* \left< n ( \la ) \right>$ as $A_{1,n-1}$-module must be of the shape $L_{\la_{(j)}} \left< d \right>$ for $1 \le j \le \ell ( \la )$  and $d \ge 0$. The module $L_{\la_{(j)}} \left< d \right>$ arises as the restriction of a (graded) $\mathfrak S_n$-module $L_{\mu} \left< d \right>$ ($\mu \in \cP_n$) such that $\la_{(j)} = \mu_{(k)}$ for $1 \le k \le \ell ( \mu )$. In case $\mu = \la$, then $[R_\la^* \left< n ( \la ) \right> : L_\la]_q = 1$ forces $L_{\la_{(j)}} \left< d \right> \subset L_\la \subset \mathsf{hd} \, K_{\la} \subset \mathsf{hd} \, R_\la^* \left< n ( \la ) \right>$.

From this, it is enough to assume $\mu \neq \la$ to conclude that $L_{\la_{(j)}} \left< d \right>$ does not yield a non-zero module of $\mathsf{hd} \, R_\la^* \left< n ( \la ) \right> / L _\la$. By Theorem \ref{GPprop} 2), we can assume $\mu > \la$. Hence, $\mu$ is obtained from $\la$ by moving one box in the Young diagram to some strictly larger entries.

In case $\mu$ is not the shape $(m^r)$, there exists $1 \le k \le \ell ( \mu )$ such that $\mu_{(k)} \neq \la_{(j)}$ for every $1 \le j \le \ell ( \la )$. It follows that $L_{\la_{(j)}} \left< d \right> \subset L _\mu \left< d \right> \subset R_\la^* \left< n ( \la ) \right>$ contains a $\mathfrak S_{n-1}$-module that is not in the head of $R_\la^* \left< n ( \la ) \right>$ as $A_{1,n-1}$-modules. Thus, this case does not occur.

In case $\mu$ is of the shape $(m^r)$, then we have $\la = ( m^{r-1},(m-1),1 )$ and $\la_{(j)} = ( m^{r-1},(m-1))$. In this case, we have $j = r+1$. In particular, grading shifts of $R_{\la_{(j)}}^*$ appears in the filtration of $R_{\la}^*$ afforded by Theorem \ref{GP} only once, and its head is a part of $L_{\la}$ by counting the degree. Therefore, $L_{\la_{(j)}} \left< d \right>$ contributes zero in $\mathsf{hd} \, R_\la^* \left< n ( \la ) \right> / L _\la$.

From these, we conclude that $\mathsf{hd} \, R_\la^* \left< n ( \la ) \right> = L_\la$ by induction hypothesis. This forces $K_\la = R_\la^* \left< n ( \la ) \right>$, and the induction proceeds.
\end{proof}

\subsection{Identification of the forms}

Consider the twisted (graded) Frobenius characteristic map
\begin{equation}
\Psi : [\cA] \ni [M] \mapsto \sum_{\mu} [M:L_\mu]_q \cdot S_{\mu} \in \La_q.\label{gFr}
\end{equation}

By Theorem \ref{GPHL}, we have
\begin{equation}
\Psi ( [K_\la] ) = Q_\la \hskip 5mm (\la \in \cP). \label{PsiK}
\end{equation}

\begin{lem}\label{tFr}
For $a, b \in \cA$, we have
$$\Psi ( \mathrm{ind} \, ( a \boxtimes b ) ) = \Psi ( a ) \cdot \Psi ( b ), \hskip 5mm \text{and} \hskip 5mm ( \Psi \times \Psi ) (\mathrm{res} \, a ) = \Delta ( \Psi ( a ) ).$$
\end{lem}

\begin{proof}
This is a straight-forward consequence of Lemma \ref{id-indres}. The detail is left to the reader.
\end{proof}

\begin{prop}\label{orth-id}
We have
$$\left< [K_\la], [K_\mu]\right>_{EP} = \left< Q_\la, Q_{\mu} \right> = \delta_{\la, \mu} b_\la.$$
In particular, we have
\begin{equation}
\left< a, b \right> _{EP} = \left< \Psi ( a ), \Psi ( b ) \right> \hskip 5mm a,b \in [\cA].\label{id-forms}
\end{equation}
\end{prop}

\begin{rem}
If we prove the identities in Corollary \ref{bigSchur} directly, then one can prove (\ref{id-forms}) without appealing to \cite{Sho02,Kat15} by Proposition \ref{idform} and its proof.
\end{rem}

\begin{proof}[Proof of Proposition \ref{orth-id}]
The equations in Theorem \ref{defHall}, that are equivalent to the Cauchy identity \cite[(4.4)]{Mac95}, are special cases of \cite[Corollary 4.6]{Sho02}. It is further transformed into the main matrix equality of the so-called Lusztig-Shoji algorithm in \cite[Theorem 5.4]{Sho02}. The latter is interpreted as the orthogonality relation with respect to $\left< \bullet, \bullet \right>_{EP}$ in \cite[Theorem 2.10]{Kat15}. In particular, Kostka polynomials defined in \cite{Mac95} and \cite{Sho02} are the same (for symmetric groups and the order $\le$ on $\cP$). This implies the first equality in view of (\ref{PsiK}). The second equality is read-off from the relation between $Q_\la$ and $Q^{\vee}_\la$. The last assertion follows as $\{ Q_\la \}_{\la \in \cP}$ forms a $\Q (\!(q)\!)$-basis of $\La_q$, and the Hall inner product is non-degenerate.
\end{proof}

\begin{prop}\label{idform}
For each $\la \in \cP$, we have $\Psi ( [P_\la] ) = s_\la$.
\end{prop}

\begin{proof}
For each $\la, \mu \in \cP$, we have
$$\delta_{\la,\mu} = \left< s_\la, S_{\mu} \right> = \left< s_{\la}, \Psi ( [L_\mu]) \right>$$
by Theorem \ref{defHall}. On the other hand, we have
$$\delta_{\la,\mu} = \mathrm{gdim} \, \mathrm{hom}_{A_n} ( P_\la, L_\mu ) = \sum_{k \ge 0} (-1)^k \mathrm{gdim} \, \mathrm{ext}^k_{A_n} ( P_\la, L_\mu ) = \left< [P_{\la}], [L_\mu] \right>_{EP}.$$

As the Hall inner product is non-degenerate (Theorem \ref{defHall}) and is the same as the Euler-Poincar\'e pairing (Proposition \ref{orth-id}), this forces $\Psi ( [P_\la] ) = s_\la$.
\end{proof}

\begin{cor}\label{bigSchur}
For each $\la \in \cP_n$, we have
\begin{align*}
s_\la & = \sum_{\mu \in \cP_n} S_{\mu} \cdot \mathrm{gdim} \, \mathrm{hom}_{\mathfrak S_n} ( L_\mu, P_\la )\\
& = \sum_{\mu \in \cP_n} S_{\mu} \cdot \mathrm{gdim} \, \mathrm{hom}_{\mathfrak S_n} ( L_\mu, L_\la \otimes \C [X_1,\ldots,X_n] )\\
& = \frac{1}{(1-q) (1-q^2) \cdots (1-q^{n})} \sum_{\mu \in \cP_n} S_{\mu} \cdot \mathrm{gdim} \, \mathrm{hom}_{\mathfrak S_n} ( L_\mu, L_\la \otimes R_{(1^n)} ).
\end{align*}
\end{cor}

\begin{proof}
In view of Proposition \ref{idform}, the first equality is obtained by just expanding $[P_\la]$ using the definition of the twisted Frobenius characteristic. The second and the third equalities follow from
$$P_{\la} \cong L_{\la} \otimes \C [X_1,\ldots,X_n] \cong L_\la \otimes R_{(1^n)} \otimes \C [X_1,\ldots,X_n]^{\mathfrak S_n}$$
as $\mathfrak S_n$-modules, where the latter isomorphism is standard (\cite{Che55}).
\end{proof}

\begin{cor}\label{idnum}
For each $M \in A_n \mathchar`-\mathsf{gmod}$, we have
$$\Psi ( [M] ) = \sum_{\la} \left< [M], [K_{\la}] \right>_{EP} Q_\la^{\vee}.$$
\end{cor}

\begin{proof}
This follows by $\Psi ( [K_\la] ) = Q_\la$, Theorem \ref{defHall}, and Proposition \ref{orth-id}.
\end{proof}

\subsection{An $\mathrm{end}$-estimate}

\begin{lem}\label{la-fix}
For each $\la \in \cP_n$, the $\mathfrak S_n$-module $L_\la$ contains a unique non-zero $\mathfrak S_{\la}$-fixed vector $($up to scalar$)$.
\end{lem}

\begin{proof}
This follows from Theorem \ref{GPprop} 2) and the Frobenius reciprocity.
\end{proof}

For each $\la \in \cP_n$, we set
\begin{align}\nonumber
A_\la & := \bigotimes_{j=1}^{\ell ( \la )} A_{\la_j} \subset A_n, \hskip 5mm \text{and}\hskip 5mm\\
\widetilde{K}^+_\la & := A_n \otimes_{A_{\la}} ( \widetilde{K}_{(\la_1)} \boxtimes \widetilde{K}_{(\la_2)} \boxtimes  \cdots \boxtimes \widetilde{K}_{(\la_{\ell ( \la )})}).\label{Kplus}
\end{align}

\begin{lem}\label{Ktriv}
We have $\widetilde{K}_{(n)} \cong L_{(n)} \otimes \C [Y]$, where $\C [Y]$ is the quotient of the polynomial ring $\C [X_1,\ldots, X_n]$ by the submodule generated by degree one part that is complementary to $\C ( X_1 + \cdots + X_n)$ as $\mathfrak S_n$-modules.
\end{lem}

\begin{proof}
We have $P_{(n)} \cong \C [X_1,\ldots, X_n]$. Its degree one part is $L_{(n)} \oplus L_{(n-1,1)}$ as $\mathfrak S_n$-modules, and quotient out by $L_{(n-1,1)}$ yields a polynomial ring $\C [Y]$ generated by the image of $\C (X_1 + \cdots + X_n) \cong L_{(n)}$.
\end{proof}

\begin{lem}\label{udeg0}
Let $\la \in \cP_n$. We have a unique graded $A_n$-module map $\widetilde{K}_\la \rightarrow \widetilde{K}^+_\la$ of degree $0$ up to scalar.
\end{lem}

\begin{proof}
We have $( \widetilde{K}^+_\la )_0 = \mathrm{Ind}^{\mathfrak S_n}_{\mathfrak S_\la} \, \mathsf{triv}$, in which $L_{\la}$ appears without multiplicity. All the $\mathfrak S_\la$-modules appearing in $( \widetilde{K}_{(\la_1)} \boxtimes \widetilde{K}_{(\la_2)} \boxtimes \cdots )$ are trivial. It follows that $[\widetilde{K}^+_\la : L_\mu]_q \neq 0$ if and only if $[\mathrm{Ind}^{\mathfrak S_n}_{\mathfrak S_\la} \, \mathsf{triv} : L_{\mu}] \neq 0$. The latter implies $\la \le \mu$ (Theorem \ref{GPprop}). Therefore, a $\mathfrak S_n$-module map $L_\la \to ( \widetilde{K}^+_\la )_0$ extends uniquely to a graded $A_n$-module map $\widetilde{K}_\la \rightarrow \widetilde{K}^+_\la$ by the definition of $\widetilde{K}_\la$.
\end{proof}

In the setting of Lemma \ref{udeg0}, we set
$$\widetilde{K}_\la':= \mathrm{Im} ( \widetilde{K}_\la \rightarrow \widetilde{K}^+_\la ).$$
For each $1 \le j \le \ell ( \la )$, we have an endomorphism $\psi_j^\la$ on $\widetilde{K}^+_\la$ extending
$$\psi_{j}^{\la} ( \widetilde{K}_{(\la_1)} \boxtimes \cdots \boxtimes \widetilde{K}_{(\la_{\ell})}) = \widetilde{K}_{(\la_1)}\left< \delta_{j,1} \right> \boxtimes \cdots \boxtimes \widetilde{K}_{(\la_{\ell})} \left< \delta_{j, \ell} \right> \subset \widetilde{K}_\la^+,$$
obtained by the multiplication of $\C [X_1,\ldots,X_n]$. Consider the group
$$\mathfrak S ( \la ) := \prod_{j \ge 1} \mathfrak S_{m_j ( \la )}.$$
\begin{lem}
The group $\mathfrak S ( \la )$ yields automorphisms of $\widetilde{K}^+_\la$ as $A_n$-modules.
\end{lem}

\begin{proof}
The group $\mathfrak S ( \la )$ permutes $\widetilde{K}_{(\la_j)}$s in (\ref{Kplus}) in such a way the size of the factors (i.e. the values of $\la_j$) are invariant. These are $A_\la$-module endomorphisms, and hence $\widetilde{K}_\la^+$ inherits these endomorphisms as required.
\end{proof}

Let $B ( \la )$ denote the subring of $\mathrm{end}_{A_n} ( \widetilde{K}^+_\la )$ generated by $\{ \psi_j^\la \}_{j=1}^{\ell ( \la )}$. The action of $\mathfrak S ( \la )$ permutes $\psi_{i}^\la$ and $\psi_j^\la$ such that $\la_i = \la _j$. Thus, $\mathfrak S ( \la )$ acts on $B ( \la )$ as automorphisms. The invariant part $B ( \la )^{\mathfrak S ( \la )}$ is a polynomial ring.

\begin{lem}\label{Bla}
For each $\la \in \cP_n$, we have
$$\mathrm{hom}_{\mathfrak S_n} ( L_\la, B ( \la ) L _0) \stackrel{\cong}{\longrightarrow} \mathrm{hom}_{\mathfrak S_n} ( L_\la, \widetilde{K}^+_\la ),$$
where $L _\la \cong L _0 \subset ( \widetilde{K}^+_\la )_0$ is the multiplicity one copy as $\mathfrak S_n$-modules. 
\end{lem}

\begin{proof}
  By construction, $\widetilde{K}^+_\la$ is a direct sum of (grading shifts of) copies of $\mathrm{Ind}^{\mathfrak S_n}_{\mathfrak S_\la} \, \mathsf{triv}$ as a $\mathfrak S_n$-module. We have $[\mathrm{Ind}^{\mathfrak S_n}_{\mathfrak S_\la} \, \mathsf{triv}: L_\la] = 1$. The action of $B ( \la )$ preserves the $\mathfrak S_n$-isotypic part. As the action of $B ( \la )$ sends $(\widetilde{K}^+_\la)_0$ to all the contributions of $\mathrm{Ind}^{\mathfrak S_n}_{\mathfrak S_\la} \, \mathsf{triv}$ in $\widetilde{K}^+_\la$, we conclude the assertion.
\end{proof}

\begin{prop}\label{ext-est}
For each $\la \in \cP_n$, we have
$$\mathrm{gdim} \, \mathrm{end}_{A_n} ( \widetilde{K}_\la' ) = b_\la^{-1} \hskip 5mm \text{and} \hskip 5mm \mathrm{end}_{A_n} ( \widetilde{K}_\la' ) \cong B ( \la )^{\mathfrak S ( \la )}.$$
\end{prop}

\begin{proof}
Since $\widetilde{K}_\la'$ has $( \widetilde{K}_\la')_0 \cong L _\la$ as its unique simple graded quotient, $\mathrm{end}_{A_n} ( \widetilde{K}_\la' )$ is determined by the image of $( \widetilde{K}_\la')_0$. In addition, $\widetilde{K}_\la'$ is stable under the action of $\mathfrak S ( \la )$ as $L_\la \subset ( \widetilde{K}_\la^+ )_0$ is. Therefore, Lemma \ref{Bla} implies $\mathrm{end}_{A_n} ( \widetilde{K}_\la' ) \subset B ( \la )^{\mathfrak S ( \la )}$. Thus, we have the inequality $\le$ in the assertion by
\begin{align*}
b_{\la}^{-1} & = \prod_{j \ge 1} \frac{1}{(1-q) \cdots (1-q^{m_j ( \la )})} \\
& = \prod_{j \ge 1} \mathrm{gdim} \, \C [x_1,\ldots,x_{m_j (\la)}]^{\mathfrak S_{m_j ( \la )}} =  \mathrm{gdim} \, B ( \la )^{\mathfrak S ( \la )}
\end{align*}
(see Corollary \ref{bigSchur} for the second equality). We have an identification
\begin{equation}
( \widetilde{K}_{(\la_1)} \boxtimes \cdots \boxtimes \widetilde{K}_{(\la_{\ell ( \la )})}) \cong e \left( \C \mathfrak S (\la) \otimes ( \widetilde{K}_{(\la_1)} \boxtimes \cdots \boxtimes \widetilde{K}_{(\la_{\ell ( \la )})}) \right) \subset \widetilde{K}_\la,\label{Ktinv}
\end{equation}
where $e = \frac{1}{|\mathfrak S (\la)|} \sum_{w \in \mathfrak S(\la)} w$. The actions of $\psi_{1}^{\la}, \ldots, \psi_{\ell (\la)}^\la$ on the first term of (\ref{Ktinv}) are induced by the multiplication of $\C [X_1,\ldots, X_n]$. Hence, the action of $B ( \la )^{\mathfrak S ( \la ) \ltimes \mathfrak S_\la} = B ( \la )^{\mathfrak S ( \la )}$ on the first two terms of (\ref{Ktinv}) are realized by the multiplication of $\C [X_1,\ldots, X_n]$. Thus, the inequality must be in fact an equality and $\mathrm{end}_{A_n} ( \widetilde{K}_\la' ) = B ( \la )^{\mathfrak S ( \la )}$.
\end{proof}

Let us consider the image of the center $\C [X_1,\ldots,X_n]^{\mathfrak S_n}$ in $\mathrm{end}_{A_n} ( \widetilde{K}_\la )$ and $\mathrm{end}_{A_n} ( \widetilde{K}_\la' )$ by $Z ( \la )$ and $Z' ( \la )$, respectively.

\begin{lem}
For each $\la \in \cP_n$, we have a quotient map
$$\mathrm{end}_{A_n} ( \widetilde{K}_\la ) \longrightarrow \mathrm{end}_{A_n} ( \widetilde{K}_\la' )$$
as an algebra that induces a surjection $Z ( \la ) \to Z' ( \la )$. In addition, $Z' ( \la )$ is precisely the image of $\C [X_1,\ldots,X_n]^{\mathfrak S_n}$ in $\mathrm{end}_{A_n} ( \widetilde{K}_\la^+ )$.
\end{lem}

\begin{proof}
By the construction of $\widetilde{K}_\la$, we have
$$\mathrm{end}_{A_n} ( \widetilde{K}_\la ) \longrightarrow \!\!\!\!\! \rightarrow \mathrm{hom}_{A_n} ( \widetilde{K}_\la, \widetilde{K}_\la' ) \cong \mathrm{hom}_{\mathfrak S_n} ( L_\la, \widetilde{K}_\la' ).$$
In view of Proposition \ref{ext-est} and Lemma \ref{Bla}, we have
$$\mathrm{hom}_{A_n} ( \widetilde{K}_\la, \widetilde{K}_\la' ) \cong \mathrm{hom}_{\mathfrak S_n} ( L_\la, \widetilde{K}_\la' ) \cong \mathrm{end}_{A_n} ( \widetilde{K}_\la' ).$$
  This proves the first assertion, as the action of $\C [X_1,\ldots,X_n]^{\mathfrak S_n}$ on $\widetilde{K}_\la'$ factors through $\widetilde{K}_\la$. The second assertion follows as the action of $\C [X_1,\ldots,X_n]^{\mathfrak S_n}$ on $\widetilde{K}_\la^+$ factors through $B(\la)^{\mathfrak S (\la)} = \mathrm{end}_{A_n} (\widetilde{K}_\la')$.
\end{proof}

\begin{lem}\label{fg-end}
For each $\la \in \cP_n$, the algebra $\mathrm{end}_{A_n} ( \widetilde{K}_\la )$ is a finitely generated module over $\C [X_1,\ldots,X_n]^{\mathfrak S_n}$.
\end{lem}

\begin{proof}
Since we have a surjection $\mathrm{end}_{A_n} ( P_\la ) \to \mathrm{end}_{A_n} ( \widetilde{K}_\la )$, it suffices to see that
$\mathrm{end}_{A_n} ( P_\la )$ is a finitely generated module over $\C [X_1,\ldots,X_n]^{\mathfrak S_n}$. We have
$$\mathrm{end}_{A_n} ( P_\la ) \cong \mathrm{Hom}_{\mathfrak S_n} ( L_{(n)}, \mathrm{End} _\C ( L _\la ) \otimes \C [X_1,\ldots,X_n] ).$$
The RHS is a finitely generated module over $\C [X_1,\ldots,X_n]^{\mathfrak S_n}$ as required.
\end{proof}

For two power series with integer coefficients
$$f ( q ) = \sum_{m} f_m q^m, g ( q ) = \sum_{m} g_m q^m \in \Z (\!(q)\!),$$
we say $f ( q ) \le g ( q )$ if we have $f_m \le g_m$ for every $m \in \Z$. We say $f ( q ) \ll g ( q )$ if
\begin{equation}
\lim_{m \to \infty} \frac{\sup \{ f_k \mid k \le m\}}{\sup \{ g_k \mid k \le m\}} = 0.\label{neg-def}
\end{equation}

\begin{thm}[\cite{Mat86}]\label{gineq}
Let $R$ be a finitely generated graded integral algebra with $\C = R_0$, and let $S$ be its proper graded quotient algebra. For a finitely generated graded $S$-module $M$, we have
$$\mathrm{gdim} \, M \ll \mathrm{gdim} \, R.$$ 
\end{thm}

\begin{proof}
This follows form \cite[Theorem 13.4]{Mat86} if we take into account the Krull dimension inequality $\dim \, R > \dim \, S$, and the completion with respect to the grading makes $R$ and $S$ into local rings.
\end{proof}

\begin{lem}\label{K'-eig}
For each $\la \in \cP_n$ and an algebra quotient $Z'(\la) \to \C$, the actions of $X_1,X_2,\ldots,X_n$ on $\C \otimes_{Z'(\la)} \widetilde{K}_\la'$ and $\C \otimes_{Z'(\la)} \widetilde{K}_\la^+$ have joint eigenvalues of shape
\begin{equation}
\alpha_1 = \cdots = \alpha_{\la_1}, \alpha_{\la_1 + 1} = \cdots = \alpha_{\la_1 + \la_2}, \ldots, \alpha_{n - \la_{\ell (\la )}+1} = \cdots = \alpha_n \label{a-config}
\end{equation}
up to $\mathfrak S_n$-permutation.
\end{lem}

\begin{proof}
By Lemma \ref{fg-end}, the modules $\C \otimes_{Z'(\la)} \widetilde{K}_\la'$ and $\C \otimes_{Z'(\la)} \widetilde{K}_\la^+$ must be finite-dimensional. Hence, the actions of $X_1,\ldots,X_n$ have joint eigenvalues. Their values can be read-off from (\ref{Kplus}).
\end{proof}

\begin{thm}\label{support}
For each $\la \in \cP_n$, we have
$$\mathrm{gdim} \, \ker \left ( \mathrm{end}_{A_n} ( \widetilde{K}_\la ) \to \mathrm{end}_{A_n} ( \widetilde{K}'_\la ) \right) \ll \mathrm{gdim} \, \mathrm{end}_{A_n} ( \widetilde{K}'_\la ).$$
\end{thm}

\begin{proof}
We set $Z := Z(\la)$ and $Z' := Z'(\la)$ during this proof. The specialization $\C \otimes_{Z} \widetilde{K}_\la$ with respect to a maximal ideal $\mathfrak n \subset Z$ decomposes into the generalized eigenspaces with respect to $X_1,\ldots,X_n$, whose set of joint eigenvalues in $\C$ have multiplicities $\mu_1, \mu_2,\ldots,\mu_{\ell}$ that constitute a partition $\mu$ of $n$. We have
\begin{equation}
[\C \otimes_{Z} \widetilde{K}_\la : L_{\gamma}]_{\mathfrak S_n} = 0 \hskip 5mm \la \not\le \gamma\label{Ktmult}
\end{equation}
by the definition of $\widetilde{K}_\la$ and the fact that $\mathfrak S_n$ has semi-simple representation theory. Being the cyclic $A_n$-module generator, we have $[\C \otimes_{Z} \widetilde{K}_\la : L_{\la}]_{\mathfrak S_n} \neq 0$.

We can choose a non-zero generalized eigenspace
$$M \subset \C \otimes_{Z} \widetilde{K}_\la$$
of $X_1, \ldots, X_n$ that can be regarded as an (ungraded) $A_\la$-module. We choose
\begin{equation}
L_{\mu^{[1]}} \boxtimes L_{\mu^{[2]}} \boxtimes \cdots \boxtimes L_{\mu^{[\ell]}} \subset M \hskip 5mm \ell = \ell (\mu)\label{small}
\end{equation}
as $\mathfrak S_\mu$-modules with partitions $\mu^{[1]}, \ldots, \mu^{[\ell]}$ of $\mu_1,\ldots, \mu_{\ell}$, respectively. Since each piece of the external tensor products of (\ref{small}) have distinct $X$-eigenvalues, we deduce
\begin{equation}
\mathrm{Ind}^{\mathfrak S_n} _{\mathfrak S_{\mu}} ( L_{\mu^{[1]}} \boxtimes L_{\mu^{[2]}} \boxtimes \cdots \boxtimes L_{\mu^{[\ell]}} )\hookrightarrow M.\label{ind-small}
\end{equation}
By the Littlewood-Richardson rule, the smallest label (with respect to $\le$) of $\mathfrak S_n$-module that appears in the LHS of (\ref{ind-small}) is attained by $\kappa \in \cP_n$ such that
$$m_i ( \kappa ) = \sum_{j = 1}^{\ell} m_i ( \mu^{[j]} ) \hskip 5mm i \ge 1.$$
For an appropriate choice in (\ref{small}), we attain $\kappa = \la$ by Lemma \ref{Kmult}. It follows that $\mu$ defines a division of entries of $\la$ into small groups. In view of (\ref{a-config}), the maximal ideal $\mathfrak n \subset Z$ is the pullback of a maximal ideal of $Z'$. In other words, we find that $Z$ shares with the same support as $Z'$ in $\Spec \, \C [X_1,\ldots,X_n]^{\mathfrak S_n}$.

We define graded $A_n$-modules $N_r$ ($r \ge 1$) as: 
$$N_r := \ker \left ( \mathrm{end}_{A_n} ( \widetilde{K}_\la ) \to \mathrm{end}_{A_n} ( \widetilde{K}'_\la ) \right)^r / \left( \ker \left ( \mathrm{end}_{A_n} ( \widetilde{K}_\la ) \to \mathrm{end}_{A_n} ( \widetilde{K}'_\la ) \right) \right)^{r+1}.$$
We show that each $N_r$ is supported in a proper subset of $\Spec \, Z'$. Equivalently, we show that general specializations of $\mathrm{end}_{A_n} ( \widetilde{K}_\la )$ and $\mathrm{end}_{A_n} ( \widetilde{K}'_\la )$ with respect to (\ref{a-config}) are the same. In view of the above construction of the partitions $\mu$ and $\kappa$, we have necessarily $\la = \mu$ and $\mu^{(i)} = ( \mu_i )$ for each $i \ge 1$ as otherwise smaller partitions arise. By Lemma \ref{Ktriv}, a thickening of (\ref{ind-small}) as (ungraded) $A_\la$-modules must be achieved by the actions of
\begin{equation}
X_1 + \cdots + X_{\la_1}, X_{\la_1 + 1} + \cdots + X_{\la_1 + \la_2}, \ldots, X_{n - \la_{\ell (\la )}+1} + \cdots + X_n. \label{x-config}
\end{equation}
As these are contained in the action of $B ( \la )$, we conclude that general specializations of $\mathrm{end}_{A_n} ( \widetilde{K}_\la )$ and $\mathrm{end}_{A_n} ( \widetilde{K}'_\la )$ are the same.

Therefore, Theorem \ref{gineq} implies
$$\mathrm{gdim} \, N_r \ll \mathrm{gdim} \, \mathrm{end}_{A_n} ( \widetilde{K}'_\la ) \hskip 5mm r > 0.$$
By Lemma \ref{fg-end} (and the support containment), we have only finitely many $r$ with $N_r \neq \{ 0 \}$. Again using Theorem \ref{gineq}, we conclude
$$\mathrm{gdim} \, \mathrm{end}_{A_n} ( \widetilde{K}_\la ) - \mathrm{gdim} \, \mathrm{end}_{A_n} ( \widetilde{K}'_\la ) = \sum_{r \ge 1} \mathrm{gdim} \, N_r \ll \mathrm{gdim} \, \mathrm{end}_{A_n} ( \widetilde{K}'_\la )$$
as required.
\end{proof}

\begin{prop}\label{selfex}
For each $\la \in \cP_n$, the module $\widetilde{K}_\la'$ admits a decreasing separable filtration whose associated graded is the direct sum of grading shifts of $K_\la$.
\end{prop}

\begin{proof}
Consider the submodule $\widetilde{N} \subset \widetilde{K}_\la^+$ generated by the unique copy $L_{(n)} \subset \mathrm{Ind}^{\mathfrak S_n}_{\mathfrak S_\la} \mathsf{triv} = ( \widetilde{K}_\la^+ )_0$. In view of Lemma \ref{Bla}, we find
\begin{equation}
\hom_{\mathfrak S_n} ( L_{(n)}, \widetilde{N} ) \cong B ( \la )^{\mathfrak S ( \la )} \cong \hom_{\mathfrak S_n} ( L_{\la}, \widetilde{K}_\la' ).
\end{equation}
Consequently, we have $\mathrm{end}_{A_n} ( \widetilde{N} ) = B ( \la )^{\mathfrak S ( \la )}$.

Let $N$ and $K$ be the specializations of $\widetilde{N}$ and $\widetilde{K}_\la$ with respect to a maximal ideal of $B ( \la )^{\mathfrak S ( \la )}$ such that the joint eigenvalues $\al_{\la_1},\al_{\la_1+\la_2},\ldots,\al_{\ell ( \la )}$ in Lemma \ref{K'-eig} are distinct. Let $M$ be a joint $\{X_i\}_i$-eigenspace of $K$ or $N$, that is a $\mathfrak S_\la$-module. The distinct eigenvalue condition implies
\begin{equation}
\mathrm{Ind}^{\mathfrak S_n}_{\mathfrak S_{\la}} M \subset N \hskip 3mm \text{or} \hskip 3mm \mathrm{Ind}^{\mathfrak S_n}_{\mathfrak S_{\la}} M \subset K.\label{ind-M}
\end{equation}
The $\mathfrak S_n$-module $L_\mu$ appears in $N$ or $K$ only if $L_\mu \subset \mathrm{Ind}^{\mathfrak S_n}_{\mathfrak S_{\la}} \mathsf{triv}$. Applying the Littlewood-Richardson rule to the middle term of (\ref{ind-M}), we deduce $M \cong \mathsf{triv}$. In particular, we have $[N :L_{\la} ]_{\mathfrak S_n} > 0 < [K : L_{(n)}]_{\mathfrak S_n}$. By the semi-continuity of the specializations, we deduce
\begin{equation}
[\C_0 \otimes_{B(\la)^{\mathfrak S ( \la )}} \widetilde{N} : L_{\la}] > 0, \hskip 3mm \text{and} \hskip 3mm [\C_0 \otimes_{B(\la)^{\mathfrak S ( \la )}} \widetilde{K}_\la':L_{(n)}] > 0.\label{mult-pos}
\end{equation}
From this, we conclude $\C_0 \otimes_{B(\la)^{\mathfrak S ( \la )}} \widetilde{K}_\la' \cong K_\la$. Thus, the torsion free $B ( \la )^{\mathfrak S ( \la )}$-action on $\widetilde{K}_\la' \subset \widetilde{K}_\la^+$ yields the assertion.
\end{proof}
\begin{cor}
Keep the setting of Proposition \ref{selfex}. We have $\Psi ( [\widetilde{K}'_{\la}] ) = Q^{\vee}_{\la}$.
\end{cor}

\begin{proof}
Compare Propositions \ref{ext-est} and \ref{selfex} with (\ref{QtoQv}) and (\ref{PsiK}).
\end{proof}

\begin{cor}\label{quotineq}
For each $\la \in \cP_n$ and an $A_n$-module proper quotient $M_\la$ of $\widetilde{K}_\la'$, we have $[L_\la: M_\la ]_q \ll b_\la^{-1}$.
\end{cor}

\begin{proof}
We borrow the setting of the proof of Proposition \ref{selfex}. Since $\mathsf{soc} \, K_{\la} = L_{(n)}$, we find $L_{(n)} \left< m \right> \subset \ker \, ( \widetilde{K}'_\la \to M_\la )$ for some $m \in \Z_{> 0}$. As all the copies of $L_{(n)}$ and $L_{\la}$ in $\widetilde{K}^+_\la$ are obtained by the $B (\la)$-action from unique copies at degree zero, we find $m' \in \Z_{> 0}$ such that
$$\widetilde{N} \left< m \right> \subset \widetilde{K}_\la' \hskip 3mm \text{and} \hskip 3mm \widetilde{K}_\la' \left< m' \right>\subset \widetilde{N} \hskip 3mm \text{inside $\widetilde{K}^+_\la$ as $A_n$-modules.}$$
 This forces $\widetilde{K}_\la' \left< m' + m \right> \subset \widetilde{N} \subset \widetilde{K}_\la'$ to be zero in $M_\la$. Therefore, we have
$$\mathrm{gdim} \, \mathrm{hom}_{\mathfrak S_n} ( L_{\la},  M_\la ) \le ( 1 - q^{m + m'}) \mathrm{gdim} \, \mathrm{hom}_{\mathfrak S_n} ( L_{\la},  \widetilde{K}_\la ) = ( 1 - q^{m + m'}) b_\la^{-1}.$$
This implies the assertion.
\end{proof}

\begin{cor}\label{surjmult}
Let $\la \in \cP_n$. Assume that $M$ is a graded $A_n$-module generated by the subspace
$$
  M^{\mathrm{top}} \cong \bigoplus_{j=1}^m L_\la \left< d_j \right> \subset M \hskip 5mm \text{such that} \hskip 5mm [M : L_{\mu} ]_q = \begin{cases} b_\la^{-1} \sum_{j=1}^m q^{d_j} & (\mu = \la)\\ 0 &(\mu \not\ge \la)\end{cases}.
$$
Then, we have $M \cong \bigoplus_{j=1}^m \widetilde{K}'_\la \left< d_j \right>$.
\end{cor}

\begin{proof}
We have a surjection
$$f : \bigoplus_{j=1}^m \widetilde{K}_\la \left< d_j \right> \longrightarrow \!\!\!\!\! \rightarrow M.$$
Consider the quotient $M'$ of $M$ by $\sum_{j=1}^m f ( \ker ( \widetilde{K}_\la \to \widetilde{K}_\la' ) \left< d_j \right>)$. Let $f' : \bigoplus_{j=1}^m \widetilde{K}_\la' \left< d_j \right> \to M'$ be the map induced from $f$. Let us choose a maximal subset $S \subset \{1,\ldots, m\}$ such that $\bigoplus_{j \in S} \widetilde{K}_\la' \left< d_j \right>$ injects into $M'$ by $f'$. We take the quotient  $M''$ of $M'$ by this image. Then, the image $K_j$ of $\widetilde{K}_\la' \left< d_j \right>$ ($j \not\in S$) in $M''$ under the induced map must be a proper quotient of $\widetilde{K}_\la' \left< d_j \right>$.

Suppose that $S \neq \{1,\ldots, m\}$. Corollary \ref{quotineq} and Theorem \ref{gineq} forces
  $$[M'':L_\la]_q \le \sum_{j \not\in S} [K_j:L_\la]_q \ll [ \widetilde{K}'_\la : L_\la ]_q.$$
  By Theorem \ref{support}, we have $[M : L_\la]_q - [ M' : L_\la]_q \ll b_\la^{-1}$.  Thus, we have
  $$[M:L_\la]_q - \sum_{j \in S} q^j [\widetilde{K}'_\la:L_\la]_q \ll \sum_{j \not\in S}^m q^{d_j} [\widetilde{K}'_\la:L_\la]_q,$$
  that is a contradiction by $[\widetilde{K}'_\la:L_\la]_q = b_\la^{-1}$. Therefore, we have $S = \{1,2,\ldots,m\}$. This implies that $f'$ is an isomorphism. By Propositions \ref{selfex} and \ref{char-K}, we conclude $M = M'$ by the comparison of graded multiplicities.
\end{proof}

\subsection{Proof of Theorem \ref{main}}\label{mainproof}
We prove Theorem \ref{main} and $\widetilde{K}_\la = \widetilde{K}_\la'$ ($\la \in \cP_n$) by induction on $n$. Theorem \ref{main} holds for $n = 1$ as $\cP_1 = \{ ( 1 )\}$, $P_{(1)} = \widetilde{K}_{(1)} =  \widetilde{K}_{(1)}' = \C [X]$, $K_{(1)} = \C$, and
$$\mathrm{ext}^{k}_{\C [X]}( \C [X], \C ) \cong \C^{\delta_{k,0}}.$$

We assume the assertion for all $1 \le n < n_0$ and prove the assertion for $n = n_0$. We fix $\la \in \cP_{n_0-1}$ and set
$$\mathrm{ind} ( \la ) := \mathrm{ind}_{1,n_0-1} ( \C [X] \boxtimes \widetilde{K}_\la ).$$
For each $\mu \in \cP_{n_0}$ and $k \in \Z$, Theorem \ref{fnr} implies
\begin{equation}
\mathrm{ext}^{k}_{A_{n_0}} ( \mathrm{ind} ( \la ), K_\mu^* ) \cong \mathrm{ext}^{k}_{A_{1,n_0-1}} ( \C [X] \boxtimes \widetilde{K}_\la, K_\mu^* ).\label{FR-main}
\end{equation}
Since $\C [X]$ is projective as $\C [X]$-modules, Theorem \ref{GP} implies that
\begin{equation}
\mathrm{gdim} \, \mathrm{ext}^{k}_{A_{1,n_0-1}} ( \C [X] \boxtimes \widetilde{K}_\la, K_\mu^* ) \cong \begin{cases} \sum_{1 \le j \le \ell ( \mu ), \la = \mu_{(j)}} q^{n ( \mu ) - n ( \mu _{(j)} ) + j} & (k=0)\\ 0 & (k \neq 0)\end{cases}\label{extcalc-main}
\end{equation}
by the short exact sequences associated to (\ref{GPfilt}). In other word, we have
$$\mathrm{gdim} \, \mathrm{hom}_{A_{1,n_0-1}} ( \C [X] \boxtimes \widetilde{K}_\la, K_\mu^* ) = q^{\star} [m_{j} ( \mu )]_q.$$
and it is nonzero if and only if $\mu_{(j)} = \la$ for some $1 \le j \le \ell ( \mu )$. This is equivalent to $\la^{(j)} = \mu$ for some $1 \le j \le \ell ( \la ) + 1$. We set $S := \{\la ^{(j)}\}_{j = 1}^{\ell ( \la )+1} \subset \mathcal P_{n_0}$.

Note that $L_\mu =\mathsf{soc}\,K_{\mu}^*$, and hence every $0 \neq f \in \mathrm{hom}_{A_{n_0}} ( \mathrm{ind} ( \la ), K_\mu^* )$ satisfies $[\mathrm{Im}\,f : L_{\mu}]_q \neq 0$. In view of Lemma \ref{Kmult}, we further deduce $[\mathrm{Im}\,f : L_{\mu}] = 1$. Therefore, the image of the map
$$f^+ : \mathrm{ind} ( \la ) \longrightarrow \left( K_{\mu}^* \right)^{\oplus \star}$$
obtained by taking the sum of all the maps of $\mathrm{hom}_{A_{n_0}} ( \mathrm{ind} ( \la ), K_\mu^* )$ satisfies
\begin{itemize}
\item $\mathsf{soc}\, \mathrm{Im} \, f^+$ is the direct sum of $L_{\mu} \left< m \right>$ ($m \in \Z$);
\item $\dim\, ( \mathsf{soc}\, \mathrm{Im} \, f^+ ) = ( \dim \, L_{\mu} ) \cdot ( \dim \, \mathrm{hom}_{A_{n_0}} ( \mathrm{ind} ( \la ), K_\mu^* ) )$. 
\end{itemize}

We consider an $A_{n_0}$-submodule generated by the preimage of $( \mathsf{soc}\, \mathrm{Im} \, f^+ )$ (considered as the direct sum of grading shifts of $L_{\mu}$), that we denote by $N_{\mu}$. Although the module $N_{\mu}$ might depend on the choice of a lift, the number of its $A_{n_0}$-module generators is unambiguously determined.

We have $\la^{(j)} \ge \la^{(j+1)}$ for $1 \le j \le \ell ( \la )$ by inspection. In particular, $S$ is a totally ordered set with respect to $\le$. Moreover, $\mathrm{ind} ( \la )$ is generated by $\mathrm{Ind}_{1,n_0-1} L _\la$ as an $A_{n_0}$-module, and an irreducible constituent of $\mathrm{Ind}_{1,n_0-1} L _\la$ is of the form $L_{\la^{(j)}}$ for $1 \le j \le ( \ell ( \la ) + 1 )$ by the Littlewood-Richardson rule. As a consequence, we find that $\sum_{\gamma \in S} N_{\gamma} = \mathrm{ind} ( \la )$. For each $1 \le j \le \ell ( \la ) + 1$, we set $N ( j ) := \sum_{i \ge j} N_{\la^{(i)}}$. We have $N ( j + 1 ) \subset N ( j )$ for $1 \le j \le \ell ( \la )$ and $N ( 1 ) = \mathrm{ind} ( \la )$. 

By the Littlewood-Richardson rule and Lemma \ref{Kmult}, we find that
\begin{equation}
[ \mathrm{ind} ( \la ) : L_{\gamma}]_q \neq 0\hskip 5mm \text{only if} \hskip 5mm \gamma \ge \la^{(\ell ( \la ) + 1 )}.\label{LRest1}
\end{equation}

\begin{claim}\label{Njmult}
We have $[N ( j ) / N ( j + 1 ) : L_{\gamma}]_q = 0$ for $\gamma < \la^{(j)}$.
\end{claim}

\begin{proof}
Assume to the contrary to deduce contradiction. We have some $1 \le j \le \ell ( \la )$ and $\gamma < \la^{(j)}$ such that $[N ( j ) / N ( j + 1 ) : L_{\gamma}]_q \neq 0$. Here we have $\la^{(\ell ( \la ) + 1 )} \le \gamma < \la^{(j)}$ by (\ref{LRest1}). By rearranging $j$, we assume that $j$ is the minimal number with this property. In particular, we have
\begin{equation}
[N(l) / N(l+1) : L_{\gamma}]_q = 0 \hskip 5mm \gamma < \la^{(l)} \hskip 3mm \text{for} \hskip 3mm l < j.\label{Nmult}
\end{equation}
This in turn implies that $[N ( l ) / N ( j ) : L _{\gamma} ]_q = 0$ for $\gamma < \la^{(j)}$ for every $l \le j$. By rearranging $\gamma$ if necessary, we can assume that the $A_{n_0}$-submodule $N^- ( j ) \subset N ( j ) / N ( j + 1 )$ generated by $\mathfrak S_{n_0}$-isotypic components $L _{\kappa}$ such that $\kappa < \la^{(j)}$ satisfies $L _{\gamma} \left< m \right> \subset \mathsf{hd} \, N^- ( j )$ and the value $m$ is minimum among all $\gamma < \la^{(j)}$. Then, the lift of  $L _{\gamma} \left< m \right> \subset \mathsf{hd} \, N^- ( j )$ to $N^- ( j )$ is uniquely determined as graded $\mathfrak S_{n_0}$-module. It follows that the maximal quotient $L_{\gamma}^+$ of $N ( j ) / N ( j + 1 )$ (and hence also a quotient of $N ( j )$) such that $\mathsf{soc} \, L_{\gamma}^+ = L_{\gamma} \left< m \right>$ is finite-dimensional (as the grading must be bounded) and $[L_{\gamma}^+ : L_{\kappa}]_q = 0$ if $\kappa < \gamma (< \la^{(j)})$. By Proposition \ref{char-K} and Theorem \ref{ad}, we find
$$\mathrm{ext}^1_{A_{n_0}} ( \mathrm{coker} \, ( L_\gamma \to L_{\gamma}^+ ), K_{\gamma}^* ) = 0$$
by a repeated applications of the short exact sequences. In particular, the non-zero map $L_{\gamma} \left< m \right> \to K_{\gamma}^* \left< m \right>$ prolongs to $L_{\gamma}^+$, and hence it gives rise to a map $N ( j ) \rightarrow K_{\gamma}^* \left< m \right>$. By (\ref{Nmult}), we additionally have
$$\mathrm{ext}^1_{A_{n_0}} ( \mathrm{ind} ( \la ) / N(j) , K_{\gamma}^* ) = 0.$$
Therefore, we deduce a non-zero map $\mathrm{ind} ( \la ) \rightarrow K_{\gamma}^* \left< m \right>$ from our assumption that does not come from the generator set of $N_{\la^{(l)}}$ for every $l$. This is a contradiction, and hence we conclude the result.
\end{proof}

We return to the proof of Theorem \ref{main}. Note that Claim \ref{Njmult} guarantees that $N ( j )$ ($1 \le j \le \ell (\la + 1)$) is defined unambiguously as all the possible generating $\mathfrak S_{n_0}$-isotypical components of $N ( j )\subset\mathrm{ind} ( \la )$ (i.e. $L_{\la^{(k)}}$ for $j \le k \le \ell ( \la ) + 1$) must belong to $N ( j )$. In view of the above and Corollary \ref{idnum}, we deduce
\begin{align}\nonumber
\Psi ( [\mathrm{ind} ( \la )]) & = \sum_{\gamma \in \cP} Q^{\vee}_\gamma \cdot \left< [\mathrm{ind} ( \la )], [K_{\gamma}] \right>_{EP}\\\nonumber
& = \sum_{\gamma \in \cP, k \in \Z} (-1)^k Q^{\vee}_\gamma \cdot \mathrm{gdim} \, \mathrm{ext}_{A_{n_0}}^k ( \mathrm{ind} ( \la ), K_{\gamma}^* )^*\\
\nonumber
& = \sum_{\gamma \in S} Q^{\vee}_\gamma \cdot \mathrm{gdim} \, \mathrm{hom}_{A_{n_0}} ( \mathrm{ind} ( \la ), K_{\gamma}^* )^*\\
& = \sum_{\gamma \in S} b_\gamma^{-1} \cdot Q_\gamma \cdot \mathrm{gdim} \, \mathrm{hom}_{A_{n_0}} ( \mathrm{ind} ( \la ), K_{\gamma}^* )^* \in \La_q.\label{indchar}
\end{align}
This expansion exhibits positivity (as a formal power series in $\Q (\!(q)\!)$).

\begin{claim}\label{direct}
For each $1 \le j \le \ell ( \la )$, the module $N (j) / N (j+1)$ is the direct sum of grading shifts of $\widetilde{K}'_{\la^{(j)}}$.
\end{claim}

\begin{proof}
We assume that the assertion holds for all the larger $j$ (or $j = \ell ( \la )+1$), and $\la^{(j)} \neq \la^{(j+1)}$ (and hence $\la^{(j)} > \la^{(j+1)}$). We apply Claim \ref{Njmult}, and compare Lemma \ref{Kmult} and Theorem \ref{GPHL} with (\ref{indchar}) to find
$$\left[\frac{\mathrm{ind} ( \la )}{N (j+1)} :L_{\la^{(j)}}\right]_q = \left[\frac{N ( j )}{N (j+1)} :L_{\la^{(j)}}\right]_q = b_{\la^{(j)}}^{-1} \cdot \mathrm{gdim} \, \mathrm{hom}_{A_{n_0}} ( \mathrm{ind} ( \la ), K_{\la^{(j)}}^* )^*.$$
Since $\Psi ( [\mathrm{ind} ( \la ) / N ( j + 1 )])$ must be the sum of $Q^{\vee}_{\gamma}$ for $\gamma = \la^{(k)}$ ($k \le j + 1$) by the induction hypothesis and the above formulae, Theorem \ref{GPHL} implies
$$[N ( j ) / N (j+1) :L_{\mu}]_q = 0 \hskip 5mm \text{if}\hskip 5mm \mu \not\ge \la^{(j)}.$$
It follows that $N ( j ) / N (j+1)$ admits a surjection from direct sum of $\widetilde{K}_{\la^{(j)}}$ with its multiplicity $\mathrm{gdim} \, \mathrm{hom}_{A_{n_0}} ( \mathrm{ind} ( \la ), K_{\la^{(j)}}^* )^*$ (as this latter number counts the number of generators of $N ( j ) / N (j+1)$). Applying Corollary \ref{surjmult}, we conclude that $N (j) / N (j+1)$ is the direct sum of grading shifts of $\widetilde{K}'_{\la^{(j)}}$. These proceed the induction, and we conclude the result.
\end{proof}

\begin{claim}\label{ind-filt}
Let us enumerate as $S = \{\gamma_1 < \gamma_2 < \cdots < \gamma_s \}$. We have a finite increasing filtration
$$\{0\} = G_0 \subset G_1 \subset G_2 \subset \cdots \subset G_{s} = \mathrm{ind} ( \la )$$
as $A_{n_0}$-modules such that each $G_i / G_{i-1}$ is isomorphic to the direct sum of grading shifts of $\widetilde{K}'_{\gamma_i}$. In addition, each $G_s / G_{i-1}$ contains a copy of $\widetilde{K}'_{\gamma_i}$ as its $A_{n_0}$-module direct summand.
\end{claim}

\begin{proof}
The first part is a rearrangement of Claim \ref{direct}.

We have $L_{\gamma_i} \subset \mathrm{Ind}^{\mathfrak S_{n_0}}_{\mathfrak S_{n_0-1}} L_\la$ as $\mathfrak S_{n_0}$-modules. If we have $[G_{s} / G_{i-1} : L_{\mu}]_q \neq 0$, then Claim \ref{direct} implies $[\widetilde{K}'_{\gamma_j} : L_{\mu}]_q \neq 0$ for some $i \le j \le s$. By Lemma \ref{Kmult}, we conclude that $\mu \ge \gamma_i$. Since $\mathrm{Ind}^{\mathfrak S_{n_0}}_{\mathfrak S_{n_0-1}} L_\la = \mathrm{ind} ( \la )_0$, we find a degree zero copy of $L_{\gamma_i}$ in $\mathsf{hd} \, \mathrm{ind} ( \la )$. By Propositions \ref{selfex} and \ref{char-K}, it must lift to a direct summand $\widetilde{K}'_{\gamma_i} \subset G_{s} / G_{i-1}$. This implies the second assertion.
\end{proof}

\begin{claim}
For each $\gamma \in S$, we have
\begin{equation}
\mathrm{ext}^{k}_{A_n} ( \widetilde{K}'_{\gamma}, K_{\mu}^* ) = \begin{cases} \C & ( k=0, \gamma = \mu )\\ \{ 0 \} & (\text{else}) \end{cases}.\label{1stext}
\end{equation}
\end{claim}
\begin{proof}
We prove (\ref{1stext}) and
\begin{equation}
\mathrm{ext}^{>0}_{A_n} ( G_s/G_{j}, K_{\mu}^* ) = 0 \label{Gext}
\end{equation}
for $0 \le j \le s$ by induction. The $j=0$ case of (\ref{Gext}) follows by (\ref{FR-main}). The $j=i-1$ case of (\ref{Gext}) implies (\ref{1stext}) for $\gamma = \gamma_i$ and $k > 0$ as $G_s / G_{i-1}$ contains $\widetilde{K}'_{\gamma_i}$ as its direct summand by Claim \ref{ind-filt}. We have
\begin{equation}
\mathrm{hom}_{A_{n_0}} ( \widetilde{K}'_{\gamma}, K_{\mu}^*) = \begin{cases} \C & (\gamma=\mu)\\ 0 & (\gamma \neq \mu)\end{cases}\label{KK-hom}
\end{equation}
by Lemma \ref{Kmult}, $\mathsf{hd} \, \widetilde{K}_{\gamma} = L_{\gamma}$, and $\mathsf{soc} \, K_{\mu}^* = L_{\mu}$. By counting the multiplicities of $L_{\gamma_i}$, we deduce
\begin{equation}
\mathrm{hom}_{A_{n_0}} ( G_s / G_{j-1}, K_{\gamma_j}^*) \stackrel{\cong}{\longrightarrow} \mathrm{hom}_{A_{n_0}} ( ( \widetilde{K}'_{\gamma_j} )^{\oplus \star}, K_{\gamma_j}^*)\label{homeqind}
\end{equation}
for $0 \le j \le s$ from Claim \ref{ind-filt}.

Now a part of the long exact sequence
\begin{align*}
0 \rightarrow & \, \mathrm{hom}_{A_{n_0}} ( G_s/G_{i}, K_{\mu}^*) \rightarrow \mathrm{hom}_{A_{n_0}} ( G_s/G_{i-1}, K_{\mu}^*) \stackrel{\cong}{\longrightarrow} \mathrm{hom}_{A_{n_0}} ( ( \widetilde{K}'_{\gamma_i} )^{\oplus \star}, K_{\mu}^*)\\
\rightarrow & \, \mathrm{ext}^1_{A_{n_0}} ( G_s/G_{i}, K_{\mu}^*) \rightarrow \mathrm{ext}^1_{A_{n_0}} ( G_s/G_{i-1}, K_{\mu}^*) = 0
\end{align*}
associated to the short exact sequence
$$0 \rightarrow ( \widetilde{K}' _{\gamma_i} ) ^{\oplus \star} \rightarrow G_s/G_{i-1} \rightarrow G_s / G_{i} \rightarrow 0,$$
as well as (\ref{KK-hom}) and (\ref{homeqind}), yields (\ref{1stext}) for $\gamma = \gamma_i$ and (\ref{Gext}) for $j=i$ from (\ref{Gext}) for $j=i-1$ inductively on $i$.
\end{proof}

We return to the proof of Theorem \ref{main}. All elements of $\cP_{n_0}$ appear as $\la^{(j)}$ for suitable $\la \in \cP_{n_0-1}$ and $1 \le j \le ( \ell ( \la ) + 1 )$. By rearranging $\la$ if necessary, we conclude (\ref{1stext}) for every $\gamma \in \cP_{n_0}$. A repeated use of short exact sequences decomposes $\{ K_{\mu} \}_{\gamma \le \mu}$ into $\{ L_{\mu} \}_{\gamma \le \mu}$ by starting from $K_{(n)} = L_{(n)}$ (see Lemma \ref{Kmult}). Substituting these to the second factor of (\ref{1stext}), we deduce
$$\mathrm{ext}^{>0}_{A_n} ( \widetilde{K}'_{\gamma}, L_{\mu} ) \neq 0 \hskip 3mm \text{implies} \hskip 3mm \mu < \gamma.$$
This implies $\widetilde{K}'_{\gamma} = \widetilde{K}_{\gamma}$ for all $\gamma \in \cP_{n_0}$. Therefore, Proposition \ref{ext-est} and Proposition \ref{selfex} imply Theorem \ref{main} 1) and 2) for $n = n_0$, and (\ref{1stext}) is Theorem \ref{main} 3) for $n = n_0$. 

\medskip

In view of the above arguments, we find that each $\mathrm{ind} ( \la )$ $(\la \in \cP_{n_0-1})$ admits a $\Delta$-filtration. Since $\mathrm{ind}_{1,\star}$ preserves projectivity, we deduce that $A_{n_0}$ admits a filtration by $\mathrm{ind} ( \la )$ $(\la \in \cP_{n_0-1})$ by the induction hypothesis. Therefore, $A_{n_0}$ admits a $\Delta$-filtration. Since each $\widetilde{K}_\la$ is generated by its simple head, applying an idempotent does not separate them out non-trivially. Therefore, we conclude that each projective module of $A_{n_0}$ also admits a $\Delta$-filtration. Given this and Theorem \ref{main} 2) and 3), the latter assertion of Theorem \ref{main} 4) is standard (see e.g. \cite[Corollary 3.12]{Kat17}). This is Theorem \ref{main} 4) for $n = n_0$. 

This completes the proof of Theorem \ref{main}.

\subsection{Applications of Theorem \ref{main}}\label{indproof}

Note that $A_n$ is a Noetherian ring as a finitely generated $A_n$-module is also finitely generated by $\C [X_1,\ldots,X_n]$. The global dimension of $A_n$ is finite (Theorem \ref{gdim}). We have $\mathrm{gdim} \, A_n \in \Z [\![q]\!]$ by inspection.

We introduce a total order $\prec$ on $\cP_n$ that refines $\le$ and set $\mathbf{e}_{\la} := \sum_{\la \succ \mu \in \cP_n} e_{\mu}$ for each $\la \in \cP_n$. The two sided ideals $A_n \mathbf{e}_{\la} A_n \subset A_n$ satisfies $A_n \mathbf{e}_{\la} A_n \subset A_n \mathbf{e}_{\mu} A_n$ if $\mu \succ \la$. By Lemma \ref{Kmult}, we deduce that
$$( A_n \mathbf{e}_\la A_n ) \otimes_{A_n} P_\la \longrightarrow \widetilde{K}_\la$$
is a surjection. By Proposition \ref{char-K} and Theorem \ref{main} 2), we further deduce
$$( A_n \mathbf{e}_\la A_n ) \otimes_{A_n} P_\la \stackrel{\cong}{\longrightarrow} \widetilde{K}_\la.$$

Theorem \ref{main} 1) implies that $\mathrm{end}_{A_n} ( \widetilde{K}_\la )$ is a graded polynomial ring for each $\la \in \cP_n$. In conjunction with Theorem \ref{main} 2), we find that
$$\mathrm{end}_{A_n} ( P_\mu, \widetilde{K}_\la )$$
is a free module over $\mathrm{end}_{A_n} ( \widetilde{K}_\la )$ for each $\la,\mu \in \cP_n$. In particular, that the graded algebra $A_n$ is an affine quasi-hereditary in the sense of \cite[Introduction]{Kle14} with $\Delta_\la = \widetilde{K}_\la$ and $\overline{\nabla}_\la = K_\la^*$ $(\la \in \cP_n)$.

\begin{thm}[\cite{Kle14} Theorem 7.21 and Lemma 7.22]\label{crit-filt}
A module $M \in A \mathchar`-\mathsf{gmod}$ admits a $\Delta$-filtration if and only if
$$\mathrm{ext}^1_{A_n} ( M, K_{\la}^* ) = 0 \hskip 5mm \la \in \cP_n.$$
A module $M \in A \mathchar`-\mathsf{fmod}$ admits a $\overline{\Delta}$-filtration if and only if
$$\mathrm{ext}^1_{A_n} ( \widetilde{K}_{\la}, M^* ) = 0 \hskip 5mm \la \in \cP_n.$$
\end{thm}

\begin{cor}[\cite{Kle14} \S 7, particularly Lemma 7.5]\label{mult}
Let $M \in A \mathchar`-\mathsf{gmod}$. If $M$ admits a $\Delta$-filtration, then the multiplicity space of $\widetilde{K}_\la$ in $M$ is given by
$$\mathrm{hom}_{A_n} ( M, K_\la )^*.$$
If the module $M$ admits a $\overline{\Delta}$-filtration, then the multiplicity space of of $K_\la$ in $M$ is given by
$$\mathrm{hom}_{A_n} ( \widetilde{K}_\la, M^* )^*.$$
\end{cor}

\begin{thm}\label{indres}
Fix $n \ge 0$, and $0 \le r \le n$. Let $\la \in \cP_n, \mu \in \cP_r, \nu \in \cP_{n-r}$. We have the following:
\begin{enumerate}
\item {\rm(Garsia-Procesi \cite{GP92})} The module $\mathrm{res}_{r,n-r} \, K_\la$ admits a $\overline{\Delta}$-filtration;
\item The module $\mathrm{ind}_{r,n-r} \, ( P_{\mu} \boxtimes\widetilde{K}_\nu )$ admits a $\Delta$-filtration.
\end{enumerate}
\end{thm}

\begin{rem}\label{HLrem}
One cannot swap the roles of $\{\widetilde{K}_\la\}_\la$ and $\{K_{\la}\}_\la$ in Theorem \ref{indres}. In fact, the polynomiality claim in Corollary \ref{HLnum} 2) is already nontrivial (without a prior knowledge of characters).
\end{rem}

\begin{proof}[Proof of Theorem \ref{indres}]
We prove the first assertion. By the second part of Theorem \ref{crit-filt}, it suffices to check the $\mathrm{ext}^1$-vanishing with respect to $L_{\mu}\boxtimes\widetilde{K}_{\nu}$ ($\mu \in \cP_r, \nu \in \cP_{n-r}$) as a module over $\C \mathfrak S_{r} \boxtimes A_{n-r}$ (equivalently, we can check the $\mathrm{ext}^1$-vanishing with respect to $P_{\mu} \boxtimes \widetilde{K}_{\nu}$ as a module of $A_{r,n-r}$; see below). In particular, we do not need to mind the first factor as the $\mathfrak S_r$-action is granted by construction. Therefore, the first assertion is just a $r$-times repeated application of Theorem \ref{GP}.

We prove the second assertion. For each $\la \in \cP_r, \mu \in \cP_{n-r}$ and $\nu \in \cP_{n}$, we have
\begin{equation}
\mathrm{ext}^{\bullet}_{A_{n}} ( \mathrm{ind}_{r,n-r} ( P_{\la} \boxtimes \widetilde{K} _\mu ), K_\nu^* ) \cong \mathrm{ext}^{\bullet}_{A_{r,n-r}} ( P_{\la} \boxtimes \widetilde{K} _\mu, K_\nu^* )\label{Ind1}
\end{equation}
by Theorem \ref{fnr}. Applying Theorem \ref{GP} to $K_{\nu}^*$ as many as $r$-times, we find that the restriction of $K_{\nu}$ to $A_{n-r}$ admits a filtration whose associated graded is the direct sum of grading shifts of $\{ K_{\gamma} \}_{\gamma \in \cP_{n-r}}$. Since $P_{\la}$ is free over a polynomial ring of $r$-variables, we have
$$\mathrm{ext}^{\bullet}_{A_{r,n-r}} ( P_{\la} \boxtimes \widetilde{K} _\mu, K_\nu^* ) \cong \mathrm{ext}^{\bullet}_{\C \mathfrak S_r \boxtimes A_{n-r}} ( L_{\la} \boxtimes \widetilde{K} _\mu, K_\nu^* ).$$
Thus, we derive a natural isomorphism
\begin{equation}
    \mathrm{ext}^{1}_{\C \mathfrak S_r \boxtimes A_{n-r}} ( L_{\la} \boxtimes \widetilde{K} _\mu, K_\nu^* ) \stackrel{\cong}{\longrightarrow} \mathrm{hom}_{\mathfrak S_r} ( L_{\la}, \mathrm{ext}_{A_{n-r}}^1 ( \widetilde{K} _\mu, K_\nu^* ) ).\label{pie}
\end{equation}
By Theorem \ref{main} 3) and Theorem \ref{GP}, the RHS of (\ref{pie}) is zero. By the first part of Theorem \ref{crit-filt}, we conclude the second assertion.
\end{proof}

\begin{cor}\label{HLnum}
Let $\la,\mu \in \cP$. We have the following:
\begin{enumerate}
\item We have $\Delta \, ( Q_\la ) \in \sum_{\gamma, \kappa} \Z_{\ge 0} [q]\,( S_\gamma \otimes Q_{\kappa} )$;
\item We have $s_{\la} \cdot Q^{\vee}_\mu \in \sum_{\gamma} \Z_{\ge 0} [q]\,Q^{\vee}_\nu$. In case $\la = (1^n)$, it is the Pieri rule.
\end{enumerate}
\end{cor}

\begin{proof}
Apply the twisted Frobenius characteristic to Theorem \ref{indres} using Lemma \ref{tFr}. Here the equality $s_{(1^n)} = Q^{\vee}_{(1^n)}$ is in \cite[I\!I\!I (2.8)]{Mac95} and the Pieri rule is in \cite[I\!I\!I (3.2)]{Mac95}.
\end{proof}

\begin{cor}\label{skewpos}
The skew Hall-Littlewood $Q$-function $Q_{\la/\nu}$ expands positively with respect to the big Schur function. In addition, we have a graded $A_{|\la| - |\nu|}$-module defined as
$$\mathrm{hom}_{A_{|\nu|}} ( \widetilde{K}_{\nu}, K_{\la}^* )^*,$$
such that its image under $\Psi$ is $Q_{\la/\nu}$.
\end{cor}

\begin{proof}
Let $\la \in \cP_n$. The Hall-Littlewood $Q$-polynomial corresponds to the module $K_\la$ by Theorem \ref{GPHL}. Therefore, its restriction admits a $\overline{\Delta}$-filtration. In particular, we have
$$[\mathrm{res}_{r,n-r} \, K_\la] = \sum_{\mu,\nu} c_\la^{\mu,\nu} [L_\mu \boxtimes K_\nu ] \hskip 5mm c_{\la}^{\mu,\nu} \in \Z_{\ge 0} [q].$$
Here $\mathrm{res}_{r,n-r}$ corresponds to $\Delta$ in $\Lambda_q$ by Lemma \ref{tFr}. Taking \cite[I\!I\!I (5.2)]{Mac98}, (\ref{Horth}), and (\ref{PsiK}) into account, we conclude that
$$Q_{\la/\nu} = \sum_{\mu} c_\la^{\mu,\nu} \Psi ( [L_\mu] ).$$
This is the first assertion by $\Psi ( [L_\mu] ) = S_\mu$, read off from (\ref{gFr}). In view of Theorem \ref{indres} 1) and Corollary \ref{mult}, we conclude the second assertion.
\end{proof}

\medskip

{\small
{\bf Acknowledgement:} The author presented an intensive lecture course based on the contents of this paper at Nagoya University on Fall 2020. The author thanks the hospitality of Shintarou Yanagida and Nagoya University. He also thanks Kota Murakami for pointing out some typos. This research was supported in part by JSPS KAKENHI Grant Number JP19H01782.}

\printbibliography
\end{document}